\documentclass[reqno,11pt]{amsart}
\usepackage{cases}
\usepackage{mathrsfs}
\usepackage[T1]{fontenc}
\usepackage{mathrsfs}
\usepackage{amsmath,latexsym,amssymb,amsfonts,amsbsy, amsthm}
\usepackage[usenames]{color}
\usepackage{xspace,colortbl}
\usepackage{epsfig}
\usepackage{amsmath,amsfonts,amsthm,amssymb,amscd}
\input amssym.def
\input amssym.tex
\newcommand{\R }{\mathbb{R} }

\newcommand{\LC}{\left(}
\newcommand{\RC}{\right)}

\newtheorem{remark}{Remark}[section]
\newtheorem{lemma}{Lemma}[section]

\newtheorem{prop}{Proposition}[section]
\newtheorem{theorem}{Theorem}[section]

\newcommand{\beq}{\begin{equation}}
\newcommand{\eeq}{\end{equation}}
\newcommand{\ben}{\begin{eqnarray}}
\newcommand{\een}{\end{eqnarray}}
\newcommand{\beno}{\begin{eqnarray*}}
\newcommand{\eeno}{\end{eqnarray*}}

\newdimen\eqjot \eqjot = 1\jot
\def\openupeq{\openup \the\eqjot}

\begin{document}

\title [Orbital stability of two-component peakons]
{Orbital stability of two-component peakons}

\author{Cheng He}
\address{Cheng He\newline
School of Mathematics and Statistics, Ningbo University, Ningbo 315211, P. R. China}
\email{1811071003@nbu.edu.cn}
\author{Xiaochuan Liu}
\address{Xiaochuan Liu\newline
School of Mathematics and Statistics, Xi'an Jiaotong University, Xi'an 710049, P. R. China}
\email{liuxiaochuan@mail.xjtu.edu.cn}
\author{Changzheng Qu}
\address{Changzheng Qu\newline
School of Mathematics and Statistics, Ningbo University, Ningbo 315211, P. R. China}
\email{quchangzheng@nbu.edu.cn}

\begin{abstract}
We prove that the two-component peakon solutions are orbitally stable in the energy space. The system concerned here is a two-component Novikov system, which is an integrable multi-component extension of the integrable Novikov equation. We improve the method for the scalar peakons to the two-component case with genuine nonlinear interactions by establishing optimal inequalities for the conserved quantities involving the coupled structures. Moreover, we also establish the orbital stability for the train-profiles of these two-component peakons by using the refined analysis based on monotonicity of the local energy and an induction method.
\end{abstract}

\maketitle \numberwithin{equation}{section}




\section{Introduction}

In this paper, we are devoted to the orbital stability of the two-component peakon solutions  and the corresponding configuration of train-profiles. The system we are concerned with is the following integrable two-component Novikov system \cite{Li}
\begin{equation}\label{tcNK}
\left\{
\begin{aligned}
&m_t+uvm_x+(2vu_x+uv_x)m=0,\quad m=u-u_{xx},\\
&n_t+uvn_x+(2uv_x+vu_x)n=0, \quad n=v-v_{xx}.
\end{aligned}
\right.
\end{equation}
Note that this system reduces to the well-studied integrable Novikov equation \cite{HW, Nov}
\begin{equation}\label{NOV}
m_t+u^2m_x+3uu_{x}m=0, \quad m=u-u_{xx},
\end{equation}
when $v=u$.

Since the celebrated work \cite{CH} by Camassa and Holm, in which they first discovered the non-smooth peaked soliton solutions (called peakons) to the Camassa-Holm (CH) equation, the existence of peakons and multi-peakons is one of the significant properties of the integrable CH-type equations. The basic wave profile of a peakon takes a quite compact form
\begin{equation}\label{profile}
\varphi_c(x-ct)=a(c)\,e^{-|x-ct|},
\end{equation}
where $c\in \R$ is the wave speed and the amplitude $a(c)$ is a function related to $c$. Remarkably, the seemingly simple form of the peakon displays the deep relationship with some important phenomena of wave propagation in shallow water waves. Indeed, due to the discussion in \cite{cons, conesc, To}, the feature of the peakons that their profile \eqref{profile} is smooth, except at the crest where it is continuous but the lateral tangents differ, is similar to that of the so-called Stokes waves of greatest height, i.e. traveling waves of largest possible amplitude which are solutions to the governing equations for irrotational water waves. There is no closed forms available for these waves, and the peakons can capture these described essential features. It is well-understood that the CH equation
\begin{equation*}
m_t+um_x+2u_{x}m=0, \quad m=u-u_{xx},
\end{equation*}
and the Degasperis-Procesi (DP) equation
\begin{equation*}
m_t+um_x+3u_{x}m=0, \quad m=u-u_{xx},
\end{equation*}
both arise as the appropriate asymptotic approximations of the Euler equations for the free-surface shallow water waves in the moderately nonlinear regime \cite{CL}, and admit the following peakon solutions ($c>0$, and anti-peakon for $c<0$) in the line \cite{CH, CHH, CHT, DHK, Len1, LS}
\begin{equation}\label{CHpeakons}
u(t, x)=\varphi_c(x-ct)=ce^{-\left| x-ct\right|}, \qquad c\neq 0.
\end{equation}
On the other hand, the peakons and the corresponding multi-peakons admit a rich mathematical structure related to the underlying integrable features. Indeed, both the CH and DP equations are integrable equations with Lax-pair formulations, and the inverse scattering approach can be used to derive explicitly these peakons \eqref{CHpeakons} and the related multi-peakon solutions \cite{BSS, ET, LS-0}. In the past ten years, two typical CH-type integrable equations with cubic nonlinearity that support peakon dynamics attracted much attention. One is the Novikov equation \eqref{NOV}, whose peakon solutions take the form of \cite{HW}
\begin{equation*}
u(t, x)=\varphi_c(x-ct)=\sqrt{c}e^{-\left| x-ct\right|}, \qquad c> 0.
\end{equation*}
Another one is the modified Camassa-Holm (mCH) equation \cite{OR}
\begin{equation*}
m_t+\left( (u^2-u_x^2)m\right)_x=0, \quad m=u-u_{xx},
\end{equation*}
which has the following peakon structure \cite{gloq}
\begin{equation*}
u(t, x)=\varphi_c(x-ct)=\sqrt{\frac{3c}{2}}\,e^{-\left| x-ct\right|}, \qquad c> 0.
\end{equation*}
Although the physical background of the Novikov and mCH equations is not so clear, their peakon and multi-peakon dynamics are demonstrated to have several non-trivial properties in the framework of Lax integrability (see the discussion in \cite{CSz, HLS}, etc).

In this paper, special concern for these peakon solutions is the issue of their stability, which lies in the fact that they are the explicit weak solutions in the sense of distribution to the corresponding equations. The peakon equations exhibit different features in contrast with the classical integrable equations such as the KdV equation, the modified KdV equation and the Schr\"{o}dinger equation that admit the smooth solitions, especially in the study of qualitative properties related to stability and instability. There are huge number of papers to study the stability and instability of solitons for classical integrable systems. We don't attempt to exhaust all the literatures, one can refer to \cite{BEN, BON, CaLi, GSS-1, GSS-2, pava} for stability issue and \cite{MM, MMT,  PW} for the issue of asymptotic stability, as well as the references therein. Note that the peakons don't admit the classical second-order derivatives and the linearized operators at peakons appear to be degenerate. So the classical methods based on the spectral analysis are not available in the case of peakons. In an intriguing paper due to Constantin and Strauss \cite{CS}, they proved the orbital stability of peakons for the CH equation by discovering several precise optimal inequalities relating to the maximum value of the approximate solutions and the conserved quantity of quadratic form and higher-order conserved quantity (see also \cite{CM} for a variational argument and \cite{Len} for the periodic case). This approach in \cite{CS} was further developed by Dika and Molinet \cite{DM} to study orbital stability of the train of peakons to the CH equation. Orbital stability of peakons and the train of peakons of other CH-type integrable equations were investigated in \cite{Kab, LL, LLQ-1, LLQ, LLOQ, QLL}. More recently, the issue of instability of peakons for the CH equation and the Novikov equation are addressed in \cite{NP} and \cite{CP}, respectively.

Compared with the rich results of the stability or instability properties of peakons for the CH-type integrable equations of scalar form, the corresponding work for the multi-component peakon profiles is quite a few. It is worth noticing that multi-component CH-type integrable systems, that are verified to admit the peakon structures as the distributional solutions, are very rare up to now. A celebrated integrable multi-component extension to the CH equation is the so-called two-component CH system \cite{CLZ, CI, OR}
\begin{eqnarray}\label{tcCH}
\left\{
\begin{aligned}
&\; m_t+um_x+2u_xm+\rho\rho_x=0, \quad m=u-u_{xx},\\
&\rho_t+(\rho u)_x=0.
\end{aligned}
\right.
\end{eqnarray}
However, system \eqref{tcCH} does not admit the peaked solitons \cite{HNT}. The orbital stability of a kind of reduced peakon profile was studied via variational method \cite{CLLQ}. The Novikov equation \eqref{NOV} admits the following two-component integrable extension, called the Geng-Xue system \cite{GX},
\begin{equation}\label{GX}
\left\{
\begin{aligned}
&m_t+uvm_x+3vu_xm=0,\quad m=u-u_{xx},\\
&n_t+uvn_x+3uv_xn=0, \quad n=v-v_{xx},
\end{aligned}
\right.
\end{equation}
which has been paid much attention recently \cite{LiLiu, LS-1, LS-2}. Although system \eqref{GX} supports the multi-peakon structures, they are derived in the framework of Lax-pair formulation of \eqref{GX} (see \cite{LS-2}), which are not weak solutions in the sense of distribution. To the best of our  knowledge, there is no work to study orbital stability of multi-component peakons for integrable multi-component CH-type equations.

Recently, another kind of two-component integrable generalization \eqref{tcNK} of the Novikov equation \eqref{NOV} was introduced in \cite{Li} and we find that this system admits the two-component peakon structure, which are given by
\begin{equation}\label{solitons}
\big(u(t, x), v(t, x)\big)=\big(\varphi_c(x-ct), \psi_c(x-ct)\big)=\big(a\varphi(x-ct), b\psi(x-ct)\big)=\big(ae^{-\left| x-ct\right|}, be^{-\left| x-ct\right|}\big),
\end{equation}
traveling at constant speed $c=ab\neq 0$. It is demonstrated that these peakons \eqref{solitons} are indeed the weak solutions of \eqref{tcNK} in the distribution form
\begin{equation}\label{weakform}
\left\{
\begin{aligned}
&u_t+uvu_x+P_x*\LC \frac{1}{2}u_x^2v+uu_xv_x+u^2v\RC+\frac 12 P*(u_x^2v_x)=0,\\
&v_t+uvv_x+P_x*\LC \frac{1}{2}v_x^2u+vv_xu_x+v^2u\RC+\frac 12 P*(v_x^2u_x)=0,
 \end{aligned}
\right.
\end{equation}
where $P(x)=e^{-\left| x\right|}/2$ and $*$ stands for convolution with respect to the spatial variable $x\in \mathbb R$. Here a question arises: are these multi-component peakons and the corresponding train-profiles for system \eqref{tcNK} stable in the energy space?

System \eqref{tcNK} adapts the following conserved densities
\begin{equation*}
E_{0}[u,v]=\int_{\mathbb R}(mn)^{\frac 13}\,dx,
\end{equation*}
\begin{equation*}
E_{u}[u]=\int_{\mathbb{R}}\LC u^2+u_x^2\RC \, dx, \quad E_{v}[v]=\int_{\mathbb{R}}\LC v^2+v_x^2\RC \, dx, \quad H[u,v]=\int_{\mathbb{R}}\big(uv+u_xv_x\big) \, dx
\end{equation*}
and
\begin{equation*}
F[u,v]=\int_{\mathbb{R}}\LC u^2v^2+\frac{1}{3}u^2v_x^2+\frac{1}{3}v^2u_x^2+\frac{4}{3}uvu_xv_x-\frac{1}{3}u_x^2v_x^2\RC \, dx,
\end{equation*}
which will play prominent role in proving stability of peakons, while the corresponding three conserved quantities of Novikov equation \eqref{NOV} are
\begin{equation*}
H_0[u]=\int_{\mathbb R}m^{\frac 23}\,dx, \;\; E[u]=\int_{\mathbb{R}}{(u^2+u_x^2)} \, dx,\;\;
F[u]=\int_{\mathbb{R}}{\left (u^4+2u^2u_x^2-\frac{1}{3}u_x^4 \right )}\, dx.
\end{equation*}
If we choose $u=v$, then we have $H_0=E_0[u,u]$, $E[u]=E_{u}[u]=E_{v}[v]=H[u,v]$ and $F[u,v]=F[u]$. Due to the existence of the $H^1$ conservation law respectively for the $u$ and $v$-components as well as the mutual interaction conservation laws $H[u, v]$, it is expected to prove stability for the two-component Novikov system in the sense of the energy space of $H^{1}\times H^{1}$-norm. In general, a small perturbation of a solitary wave can yield another one with a different speed and phase shift. It is appropriate to define the orbit of traveling-wave solutions $(\varphi_{c}, \, \psi_{c})$ to be the set $U(\varphi, \psi)=\{(a\varphi(\cdot+x_{1}), b\psi(\cdot+x_{2})), x_{1} \in \mathbb{R}, x_{2} \in \mathbb{R}\}$. However, if $x_{1}\neq x_{2}$, the functionals $F[\varphi_{c}, \, \psi_{c}]$ and $H[\varphi_{c}, \, \psi_{c}]$ are not conserved in the time evolution for $(\varphi_{c}, \, \psi_{c})$ in this set $U(\varphi, \psi)$. Thus, we consider here a suitable orbit of the traveling-wave solutions $\varphi_{c}$ and $\psi_{c}$ to be the set $U_0(\varphi, \psi)=\{(a\varphi(\cdot+x_{0}), b\psi(\cdot+x_{0})), x_{0} \in \mathbb{R}\}$ and the peakon solutions of the two-component Novikov equation are called orbitally stable if a wave standing close to the peakon remains close to the orbit $U_0(\varphi, \psi)$ at all the later existence time.

The first main result is stated as follows. Here, we only consider the case of peakons traveling to the right, i.e. the case of $a>0$, $b>0$ and then $c=ab>0$ in \eqref{solitons}.

\begin{theorem}\label{thm1.1}
\, Let $(\varphi_c, \, \psi_c)$ be the peaked solitons in \eqref{solitons}, traveling with speed $c=ab>0$. Then $(\varphi_c, \, \psi_c)$ are orbitally stable in the following sense. Assume that $u_0,v_0 \in H^s(\mathbb{R})$ for some $s \geq 3$, $(1-\partial_x^2)u_0(x)$ and $(1-\partial_x^2)v_0(x)$ are nonnegative, and there is a $\delta>0$ such that
\begin{align*}
\left \| (u_0,v_0)-(\varphi_c,\psi_c) \right \|_{H^1(\mathbb{R}) \times H^1(\mathbb{R})}\leq \left \| u_0-\varphi_c \right \|_{H^1(\mathbb{R})} +\left \| v_0-\psi_c \right \|_{H^1(\mathbb{R})} < \delta.
\end{align*}
Then the corresponding solution $(u(t, x), v(t, x))$ of the Cauchy problem for the two-component Novikov system \eqref{tcNK} with the initial data $u(0, x)=u_0(x)$ and $v(0, x)=v_0(x)$ satisfies
\begin{eqnarray*}
\begin{aligned}
&\sup_{t \in [0,T)}\left \| \big{(}u(t, \cdot), v(t, \cdot)\big{)}-\big{(}\varphi_c(\cdot-\xi(t)), \psi_c(\cdot-\xi(t))\big{)} \right \|_{H^1(\mathbb{R}) \times H^1(\mathbb{R})}\\
&\leq \sup_{t \in [0,T)}{\left \| u(t, \cdot)-\varphi_c(\cdot-\xi(t)) \right \|_{H^1(\mathbb{R})}+\left \| v(t, \cdot)-\psi_c(\cdot-\xi(t)) \right \|_{H^1(\mathbb{R})}} < A\delta^\frac14,
\end{aligned}
\end{eqnarray*}
where $T>0$ is the maximal existence time, $\xi(t) \in \mathbb{R}$ is the maximum point of the function $u(t, x)v(t, x)$, the constant $A$ depends only on $a$, $b$ as well as the norms $\left \| u_0 \right \|_{H^s(\mathbb{R})}$ and $\left \| v_0 \right \|_{H^s(\mathbb{R})}$.
\end{theorem}

To prove orbital stability of the two-component peakons $(\varphi_c, \, \psi_c)$, some new insights are developed. We aim to obtain for each component $u$ and $v$ the dynamical estimates $|u(t, \xi(t))-a|$ and $|v(t, \xi(t))-b|$ along some trajectory $t \mapsto \xi(t)$, where $a$ and $b$ are the maximal value of the component $\varphi_c$ and $\psi_c$, respectively. Here, due to the nonlinear interaction between $u$ and $v$ involved in system \eqref{tcNK}, the key obstacle is how to find the suitable location of $\xi(t)$ in order to derive the precise estimates for $|u(t, \xi(t))-a|$ and $|v(t, \xi(t))-b|$ (note that in the case of scalar peakons such $\xi(t)$ is always chosen to locate at the maximal point of perturbed solution, which is no longer valid in the multi-component case considered here and $\xi(t)$ must change according to the appearance of characteristic speed of nonlinear interaction $uv$). Moreover, the dynamical energy identities and energy inequalities should involve the nonlinear interaction of the two components. In addition, the conservation law $F[u,v]$ is much more complicated than $F[u]$ of the Novikov equation. Therefore, the stability issue of the two-component peakon solutions is more subtle. To overcome the difficulties, two observations will be crucial. System \eqref{tcNK} not only has the separated $H^1$ conserved quantities $\int (u^2+u_x^2) dx$ and $\int(v^2+v_x^2) dx$ with which one can derive the pointwise estimates separately for each component $u$ and $v$, but also the second-order interacting conserved quantity $H[u,v]=\int (uv+u_xv_x) dx$ with which we are motivated to find the exact location of $\xi(t)$ as the maximal point of the multiplication function $u(t, \cdot)v(t, \cdot)$ of two components. This argument is quite different from the case for the scalar CH-type equations. On the other hand, new optimal energy identities for $H[u, v]$ and $F[u, v]$ are established. The new insight is that the precise control of two components is involved in one optimal energy identity. This point is also different from the scalar case. Based on these observations together with corresponding refined analysis, we are able to prove orbital stability of the two-component peakons $(\varphi_c, \, \psi_c)$ on the line ${\mathbb R}$.

\begin{remark}
For the Geng-Xue system \eqref{GX}, even though the peakon solutions in the Lax-pair sense are considered, system \eqref{GX} does not admit sufficient conserved quantities to establish the corresponding estimates for the stability.
\end{remark}

For the issue of orbital stability of train-profiles of these two-component peakons, we have the following result.

\begin{theorem}\label{trainsstable}
Let be given $N$ velocities $c_{1},c_{2},\cdots, c_{N}$ such that $0<a_{1}<a_{2}<...<a_{N}$, $0<b_{1}<b_{2}<...<b_{N}$ and $c_{i}=a_{i}b_{i}$ for any $i\in \{1,...,N\}$. There exist $A>0$, $L_0>0$ and $\epsilon_{0}>0$ such that if the initial data $(u_0,v_0) \in H^s(\mathbb{R})\times H^s(\mathbb{R})$ for some $s \geq 3$ with $(1-\partial_x^2)u_0(x)$ and $(1-\partial_x^2)v_0(x)$ being nonnegative, satisfy
\begin{align}\label{initialdata-1}
{\left \| u_0- \sum_{i=1}^N{\varphi_c(\cdot - z_i^0)} \right \|_{H^1}}+{\left \| v_0- \sum_{i=1}^N{\psi_c(\cdot - z_i^0)} \right \|_{H^1}}\leq {\epsilon}
\end{align}
for some $0<\epsilon <\epsilon_0$ and $ z_i^0-z_{i-1}^0 \geq L$ with $L>L_{0}$, then there exist ${x}_1(t),...,{x}_N(t)$ such that the corresponding strong solution $(u(t, x), v(t, x))$ satisfies
\begin{align*}
{\left \| u(t, \cdot)- \sum_{i=1}^N{\varphi_c(\cdot - {x}_i(t))} \right \|_{H^1}}+{\left \| v(t, \cdot)- \sum_{i=1}^N{\psi_c(\cdot - {x}_i(t))} \right \|_{H^1}} \leq A\LC \epsilon^{ \frac{1}{4}}+L^{- \frac{1}{8}}\RC,
\end{align*}
$\forall t \in [0,T)$, where $x_{j}(t)-x_{j-1}(t)>L/2$.
\end{theorem}

In general, two main ingredients in the proof of orbital stability for the train-proflies of peakons are involved \cite{DM,LLQ,MMT}. One is orbital stability of the single peakons, and the another one is the property of almost monotonicity of the local energy on the right hand side of the peakons. For the two-component peakons, more difficulties come from the interaction of the two components $u(t, x)$ and $v(t, x)$. The first one is how to establish inequalities among the localized conserved quantities to verify the orbital stability of the two-component peakons separately. For the train-profiles of peakons, we need to apply $N$ inequalities to control $2N$ estimates. The second one is to establish monotonicity result of the functionals $\mathcal{J}^{u,v}_{j,k}(t)$ since we can not identify the sign for the term $u_xv_x$ in the conserved density $H[u,v]$. To overcome the difficulties, we use
the conserved densities $E_u[u]$, $E_v[v]$, $H[u,v]$ and $F[u,v]$, and establish the delicate inequalities relating to the conserved densities and the maximal value of the two components $u(t,x)$ and $v(t,x)$. And in the case of train-profile of two-component peakons, we apply the proof of the single peakons, the modulation theory, the accurate estimates on the conserved densities $E_u[u]$, $E_v[v]$, $H[u, v]$ and $F[u,v]$ and the induction method to obtain the desired result.

The remainder of the paper is organized as follows. In Section 2, we provide a brief discussion on the integrability, conservation laws, the sign invariant property of $m(t, x)$ and $n(t, x)$ of system \eqref{tcNK},  and local well-posedness result on Cauchy problem of system \eqref{tcNK}. In Section 3, we prove orbital stability of single peakons on the line. Finally in Section 4, we verify orbital stability of the train of peakons.

\section{Preliminaries}

In the present section, the issue of well-posedness is discussed. First of all, we call that functions $(u,v)\in C([0,T);H^1(\mathbb R))\times C([0,T);H^1(\mathbb R))$ is a solution of system \eqref{tcNK}, if $(u,v)$ is a solution of \eqref{weakform} in the sense of distributions and $E_{u}[u]$, $E_{v}[v]$, $H[u, v]$ and $F[u, v]$ are conserved quantities.

Consider the following Cauchy problem of system \eqref{tcNK} in the whole line $\R$
\begin{equation}\label{tcNK-1}
\left\{
\begin{aligned}
&m_t+uvm_x+(2vu_x+uv_x)m=0,\quad m=u-u_{xx},\\
&n_t+uvn_x+(2uv_x+vu_x)n=0, \quad n=v-v_{xx}, \quad t>0, \; x\in \R,\\
&u(0,x)=u_0(x), \;\;v(0,x)=v_0(x), \quad x\in \R.
\end{aligned}
\right.
\end{equation}

First, similar to the results for the two-component CH system given in \cite{FQ}, one can establish the following local well-posedness result.
\begin{theorem}\label{th-2.1}
Given $z_0=(u_0, v_0)^T\in H^s(\R)\times H^s(\R) (s>3/2)$, there exists a maximal $T=T(\| z_0\|_{H^s(\R)\times H^s(\R)})>0$ and a unique strong solution $z=(u, v)^T$ to \eqref{tcNK-1} such that
\begin{equation*}
z=z(\cdot, z_0)\in C([0, T);H^s(\R)\times H^s(\R))\cap C^1([0,T); H^{s-1}(\R)\times H^{s-1}(\R)).
\end{equation*}
In addition, the solution depends continuously on the initial data, i.e. the mapping
\begin{equation*}
z_0\rightarrow z(\cdot, z_0): H^s(\R)\times H^s(\R) \rightarrow C([0, T); H^s(\R)\times H^s(\R))\cap C^1([0, T); H^{s-1}(\R)\times H^{s-1}(\R))
\end{equation*}
is continuous. Furthermore, the quantities $E_{u}[u]$, $E_{v}[v]$, $H[u, v]$ and $F[u, v]$ are all conserved along the solution $z=(u, v)^T$.
\end{theorem}

Consider the flow governed by $(uv)(t, x)$
\begin{equation}\label{firstflow}
\left\{
\begin{aligned}
&\frac{d}{dt}q(t, x)=(uv)(t, q(t, x)),\quad x\in \R, \quad t\in [0, T),\\
&q(0, x)=x,\quad x\in \R.
\end{aligned}
\right.
\end{equation}
Just as the case for the Novikov equation \eqref{NOV}, the following lemma can be proved.
\begin{lemma}\label{Flow}
Assume that $(u_{0},v_0)\in H^s(\mathbb R)\times H^s(\mathbb R)$ with $s>3/2$ and $T>0$ be the maximal existence time of the corresponding strong solution $(u, v)$ to the Cauchy problem of system \eqref{tcNK-1}. Then the problem \eqref{firstflow} has a unique solution $q(t, x)\in C^1([0, T)\times \R)$. Furthermore, the map $q(t,\cdot)$ is an increasing diffeomorphism over $\R$ with
\begin{eqnarray}\label{diffflow}
\begin{aligned}
q_{x}(t,x)=exp\left(\int^t_{0}(uv)_x(s,q(s,x))ds\right),\quad (t,x)\in [0,T)\times \mathbb R.
\end{aligned}
\end{eqnarray}
\end{lemma}

\begin{lemma}\label{nonnegative}
Let $u_0, v_0\in H^s(\mathbb R)$, $s\geq 3$. Assume that $m_0=u_0-u_{0xx}$ and $n_0=v_0-v_{0xx}$ are nonnegative in the line.
Then for the corresponding strong solution $u, v\in C([0,T); H^s(\R))\cap C^1([0,T); H^{s-1}(\R))$ of the Cauchy problem of the two-component Novikov system \eqref{tcNK-1} with the initial data $u_0,\,v_0$, we have for all $t\in [0, T)$, $m(t, x)$ and $n(t, x)$ are both nonnegative functions. In addition, $u(t, \cdot) \geq 0$, $v(t, \cdot) \geq 0$ and $|u_{x}(t, \cdot)|\leq u(t, \cdot)$, $|v_{x}(t, \cdot,)|\leq v(t, \cdot)$ on the line.
\end{lemma}
\begin{proof}
It follows from the Cauchy problem \eqref{tcNK-1} of the two-component Novikov system that along the flow \eqref{firstflow}, $m$ and $n$ satisfy
\begin{equation}\label{weakformFLOW}
\left\{
\begin{aligned}
&m'+(2vu_x+uv_x)(t,q(t,x))m(t,q(t,x))=0,\\
&n'+(2uv_x+vu_x)(t,q(t,x))n(t,q(t,x))=0,
\end{aligned}
\right.
\end{equation}
where $'$ denotes the derivative with respect to $t$ along the flow \eqref{firstflow}. Denote
\begin{eqnarray*}
\gamma_1(t,x)=exp\left(\int^t_{0}(2vu_x+uv_x)(s,q(s,x))ds\right), \;  \gamma_2(t,x)=exp\left(\int^t_{0}(2uv_x+vu_x)(s,q(s,x))ds\right).
\end{eqnarray*}
Then they satisfy
\begin{eqnarray*}
\gamma'_{1}(t,x)=(2vu_x+uv_x)(t,q(t,x))\gamma_1,\quad \gamma'_{2}(t,x)=(2uv_x+vu_x)(t,q(t,x))\gamma_2.
\end{eqnarray*}
Let
\begin{eqnarray*}
\begin{aligned}
\tilde{m}(t,x)=\gamma_{1}(t,x)m(t,q(t,x)), \quad \tilde{n}(t,x)=\gamma_{2}(t,x)n(t,q(t,x)).
\end{aligned}
\end{eqnarray*}
\eqref{weakformFLOW} become
\begin{eqnarray}
\begin{aligned}
\tilde{m}'(t,x)=0, \quad \tilde{n}'(t,x)=0.
\end{aligned}
\end{eqnarray}
The equations
\begin{eqnarray*}
\begin{aligned}
\tilde{m}(t,x)=m_0 \quad and \quad \tilde{n}(t,x)=n_0
\end{aligned}
\end{eqnarray*}
lead to
\begin{eqnarray*}
\begin{aligned}
&m(t,q(t,x))=exp\left(-\int^t_{0}(2vu_x+uv_x)(s,q(s,x))ds\right)m_0,\\
&n(t,q(t,x))=exp\left(-\int^t_{0}(2uv_x+vu_x)(s,q(s,x))ds\right)n_0.
\end{aligned}
\end{eqnarray*}
Thus, for all $t\in [0, T)$, we have $m(t, \cdot) \geq 0$, $n(t, \cdot) \geq 0$. And then $u(t, \cdot) \geq 0$, $v(t, \cdot) \geq 0$.

Formally regarding $m(x)=u(x)-u_{xx}(x)$, it holds that for all $x\in \R$,
\begin{eqnarray*}
\begin{aligned}
u(x)=\frac{e^{-x}}{2}\int^{x}_{-\infty}e^{y}m(y)dy+\frac{e^{x}}{2}\int^{\infty}_{x}e^{-y}m(y)dy
\end{aligned}
\end{eqnarray*}
and
\begin{eqnarray*}
\begin{aligned}
u_{x}(x)=-\frac{e^{-x}}{2}\int^{x}_{-\infty}e^{y}m(y)dy+\frac{e^{x}}{2}\int^{\infty}_{x}e^{-y}m(y)dy.
\end{aligned}
\end{eqnarray*}
Then we infer that
\begin{eqnarray*}
\begin{aligned}
u(x)\geq |u_{x}(x)|, \; \forall x\in \mathbb R.
\end{aligned}
\end{eqnarray*}
Similarly, we find
\begin{eqnarray*}
\begin{aligned}
v(x)\geq |v_{x}(x)|, \; \forall x\in \mathbb R.
\end{aligned}
\end{eqnarray*}
This completes the proof of the lemma.
\end{proof}

\section{Stability of two-component peakons}

In this section, we prove Theorem 1.1, which will be based on a series of lemmas. Note that the assumptions on the initial profile guarantee the existence of a unique positive solution for the Cauchy problem \eqref{tcNK-1} of the two-component Novikov system. In general, for $a>0$ and $b>0$, the profile functions of peakon solutions $\varphi_c(x)=ae^{-\left| x\right|}$ and $\psi_c(x)=be^{-\left| x\right|}$ are in $H^1(\R)$, which have peaks at $x=0$, and thus
\begin{equation*}
\max_{x \in \mathbb{R}}{\varphi_c(x)}=\varphi_c(0)=a \quad \mathrm{and} \quad \max_{x \in \mathbb{R}}{\psi_c(x)}=\psi_c(0)=b.
\end{equation*}
A direct calculation gives
\begin{equation*}
E_{u}[\varphi_c(x)]=\left \| \varphi_c \right \|^2_{H^1}=2a^2, \qquad E_{v}[\psi_c(x)]=\left \| \psi_c \right \|^2_{H^1}=2b^2
\end{equation*}
and
\begin{equation*}
H[\varphi_c(x),\psi_c(x)]=2ab, \qquad F[\varphi_c(x),\psi_c(x)]=\frac43 \, a^2b^2.
\end{equation*}

Due to the conservation of the $H^1$-norm of each component $u$ and $v$, the following pointwise identities still hold for the two-component Novikov system as in the scalar CH and Novikov cases.
\begin{lemma}\label{lem1.1}
For any $u,v \in H^1(\mathbb{R})$ and $\xi \in \mathbb{R}$, we have
\begin{eqnarray}\label{energyESI}
\begin{aligned}
&E_{u}[u]-E_{u}[\varphi_c]=\left \| u- \varphi_c(\cdot - \xi) \right \|^2_{H^1(\mathbb{R})}+4a(u(\xi)-a),\\
&E_{v}[v]-E_{v}[\psi_c]=\left \| v- \psi_c(\cdot - \xi) \right \|^2_{H^1(\mathbb{R})}+4b(v(\xi)-b).
\end{aligned}
\end{eqnarray}
\end{lemma}

In the following two lemmas, two energy identities relating some kind of critical values of $u$ and $v$ to the invariants $H[u, v]$ and $F[u, v]$ are established. Consider $0\not\equiv u,\,v\in H^{s}(\R)$, $s\geq 3$, and $u, v\geq 0$. Then $u,v\in C^2$ due to the Sobolev imbedding theory. Since $u$ and $v$ decay at infinity, it must have a point with global maximal value of the multiplication function $u(x)v(x)$. Thus, in the following, set for some $\xi\in \R$
\begin{eqnarray}\label{maxpoint}
M=\max_{x\in \R}\{u(x)v(x)\}=u(\xi)v(\xi).
\end{eqnarray}

\begin{lemma}\label{lem1.2}
Let $0\not\equiv u,\,v\in H^{s}(\R)$, $s\geq 3$ and $u, v\geq 0$. Define the functions $g_1(x)$ and $g_2(x)$ by
\begin{equation}\label{functiong1}
g_1(x)=
\left\{
\begin{aligned}
&u(x)-u_{x}(x), \quad x<\xi,\\
&u(x)+u_{x}(x), \quad x>\xi,
 \end{aligned}
\right.
\end{equation}
and
\begin{equation}\label{functiong2}
g_2(x)=
\left\{
\begin{aligned}
&v(x)-v_{x}(x), \quad x<\xi,\\
&v(x)+v_{x}(x), \quad x>\xi.
 \end{aligned}
\right.
\end{equation}
Then
\begin{eqnarray}\label{g1g2ESI}
\begin{aligned}
\int_{\mathbb{R}}{g_{1}(x)g_{2}(x)}dx=H[u,v]-2M.
\end{aligned}
\end{eqnarray}
\end{lemma}

\begin{proof}
To show \eqref{g1g2ESI}, we evaluate the integral of $g_1(x)g_2(x)$ on $\mathbb R$. Thus,
\begin{eqnarray*}
\begin{aligned}
&\int_{\mathbb{R}}{g_{1}(x)g_{2}(x)}dx\\
&=\int_{-\infty}^{\xi}{\left(u(x)-u_{x}(x)\right)\left(v(x)-v_{x}(x)\right)}dx+\int_{\xi}^{\infty}{\left(u(x)+u_{x}(x)\right)\left(v(x)+v_{x}(x)\right)}\,dx\\
&=\int_{-\infty}^{\xi}{\left(uv-(uv)_x+u_xv_x\right)}dx+\int_{\xi}^{\infty}{\left(uv+(uv)_x+u_xv_x\right)}\,dx\\
&=H[u,v]-2M.
\end{aligned}
\end{eqnarray*}
\end{proof}

The construction of the auxiliary function $h(x)$ in the following lemma is crucial in the proof of stability of peakons, which is different from the scalar cases of CH and DP equations. This new defined function is a nontrivial refinement of the case in the Novikov equation.

\begin{lemma}\label{lem1.3}
With the same assumptions and notation as in Lemma \ref{lem1.2}, define the function $h$ by
\begin{equation}\label{functionalh}
h(x)=
\left\{
\begin{aligned}
&uv-\frac{1}{3}(uv)_{x}-\frac{1}{3}u_{x}v_{x}, \quad x<\xi,\\
&uv+\frac{1}{3}(uv)_{x}-\frac{1}{3}u_{x}v_{x}, \quad x>\xi.
 \end{aligned}
\right.
\end{equation}
Then
\begin{eqnarray}\label{Finalesi}
\begin{aligned}
\int_{\mathbb{R}}{h(x)g_{1}(x)g_{2}(x)}dx=F[u,v]-\frac{4}{3}\,M^2.
\end{aligned}
\end{eqnarray}
\end{lemma}

\begin{proof}
To show \eqref{Finalesi}, we evaluate the integral of $h(x)g_{1}(x)g_{2}(x)$ on $\mathbb R$. Thus,
\begin{eqnarray*}
\begin{aligned}
&\int_{\mathbb{R}}{h(x)g_{1}(x)g_{2}(x)}dx\\
&=\int_{-\infty}^{\xi}\left[uv-\frac{1}{3}(uv)_{x}-\frac{1}{3}u_{x}v_{x}\right]\left[uv-(uv)_{x}+u_{x}v_{x}\right]\,dx\\
&\qquad+\int_{\xi}^{\infty}\left[uv+\frac{1}{3}(uv)_{x}-\frac{1}{3}u_{x}v_{x}\right]\left[uv+(uv)_{x}+u_{x}v_{x}\right]\,dx\triangleq I+\Pi.
\end{aligned}
\end{eqnarray*}
We do computation:
\begin{eqnarray*}
\begin{aligned}
I&=\int_{-\infty}^{\xi}\left[uv-\frac{1}{3}(uv)_{x}-\frac{1}{3}u_{x}v_{x}\right]\left[uv-(uv)_{x}+u_{x}v_{x}\right]\,dx\\
&=\int_{-\infty}^{\xi}\left(u^2v^2+\frac{1}{3}u^2v_x^2+\frac{1}{3}v^2u_x^2+\frac{4}{3}uvu_xv_x-\frac{1}{3}u_x^2v_x^2\right)\,dx-\frac{4}{3}\,\int_{-\infty}^{\xi}(uv)(uv)_{x}\,dx\\
&=\int_{-\infty}^{\xi}\left(u^2v^2+\frac{1}{3}u^2v_x^2+\frac{1}{3}v^2u_x^2+\frac{4}{3}uvu_xv_x-\frac{1}{3}u_x^2v_x^2\right)\,dx-\frac{2}{3}M^2.
\end{aligned}
\end{eqnarray*}
Similarly,
\begin{eqnarray*}
\begin{aligned}
\Pi=\int_{\xi}^{\infty}\left(u^2v^2+\frac{1}{3}u^2v_x^2+\frac{1}{3}v^2u_x^2+\frac{4}{3}uvu_xv_x-\frac{1}{3}u_x^2v_x^2\right)\,dx-\frac{2}{3}M^2.
\end{aligned}
\end{eqnarray*}
Combining above, we have
\begin{eqnarray*}
\begin{aligned}
&\int_{\mathbb{R}}{h(x)g_{1}(x)g_{2}(x)}dx\\
&=\int_{-\infty}^{\infty}\left(u^2v^2+\frac{1}{3}u^2v_x^2+\frac{1}{3}v^2u_x^2+\frac{4}{3}uvu_xv_x-\frac{1}{3}u_x^2v_x^2\right)\,dx-\frac{4}{3}\,M^2=F[u,v]-\frac{4}{3}\,M^2.
\end{aligned}
\end{eqnarray*}
\end{proof}

With the two energy identities \eqref{g1g2ESI} and \eqref{Finalesi} in hand, one can derive the following delicate relation between the second order conserved quantity $H[u, v]$ and the higher-order conserved quantity $F[u, v]$ for the strong solution $(u, v)$.

\begin{lemma}\label{lem1.4}
Assume that $0\not\equiv u_0,\,v_0\in H^s$, $s\geq 3$ and $m_0=u_0-u_{0xx}\geq 0$, $n_0=v_0-v_{0xx}\geq 0$. For the corresponding strong solution $(u(t, x),\,v(t, x))$ with initial data $(u_0,\,v_0)$ in the lifespan $[0,\,T)$, there holds
\begin{eqnarray}\label{functionalESI}
\begin{aligned}
F[u, v]-\frac{4}{3}M(t)H[u, v]+\frac{4}{3}M(t)^2 \leq 0, \quad \forall t\in [0,\,T),
\end{aligned}
\end{eqnarray}
where $M(t)=\max_{x\in R}\{u(t, x)v(t, x)\}=u(t, \xi(t))v(t, \xi(t))$ for some trajectory $\xi(t)\in \R$ in $[0,\,T)$.
\end{lemma}

\begin{proof}
First, by the sign-invariant property, the solution $(u(t, x),\,v(t, x))$ satisfies $m(t, x)=u(t, x)-\partial^2_{x}u(t, x)\geq 0$ and $n(t, x)=v(t, x)-\partial^2_{x}v(t, x)\geq 0$ for all $(t, x)\in [0,\,T)\times \R$. It follows that $(u(t, x),\,v(t, x))$ is positive solution and fulfills all the conditions assumed in Lemmas 3.2 and 3.3. Hence, we have the following energy identities of the dynamical forms in $[0,\,T)$
\begin{equation}\label{H(t)}
\int_{\mathbb{R}}{g_{1}(t, x)g_{2}(t, x)}dx=H[u, v]-2M(t)
\end{equation}
and
\begin{equation}\label{F(t)}
\int_{\mathbb{R}}{h(t, x)g_{1}(t, x)g_{2}(t, x)}dx=F[u, v]-\frac{4}{3}\,M(t)^2.
\end{equation}

We now claim that for any $(t, x)\in [0,\,T)\times \R$
\begin{equation*}
h(t, x)=
\left\{
\begin{aligned}
&\LC uv-\frac{1}{3} (uv)_{x}-\frac{1}{3}u_{x}v_{x}\RC(t, x), \quad x<\xi(t),\\
&\LC uv+\frac{1}{3} (uv)_x -\frac{1}{3}u_{x}v_{x}\RC(t, x), \quad x>\xi(t)
 \end{aligned}
\right.
\leq \frac{4}{3}u(t, x)v(t, x).
\end{equation*}
In fact, due to the definition of $h(t, x)$, it follows from the fact $u(t, x)\geq |u_{x}(t, x)|$ and $v(t, x)\geq |v_{x}(t, x)|$ that
\begin{equation*}
h(t, x)= \frac{4}{3}u(t, x)v(t, x)-\frac{1}{3}(u\pm u_{x}(t, x))(v\pm v_{x}(t, x))\leq \frac{4}{3}u(t, x)v(t, x).
\end{equation*}
Thus, the combination of the above inequalities yields
\begin{eqnarray*}
h(t, x)\leq \frac{4}{3}u(t, x)v(t, x)\leq \frac{4}{3}\max_{x\in \R}\{u(t, x)v(t, x)\}=\frac{4}{3}M(t), \quad \forall (t, x)\in [0,\,T)\times \R.
\end{eqnarray*}

Now, using estimates $|u_x|\leq u$ and $|v_x|\leq v$ again in the expressions \eqref{functiong1} and \eqref{functiong2}, we obtain \eqref{functionalESI} from \eqref{H(t)} and \eqref{F(t)}.
\end{proof}

In the following lemma, we study the perturbation of the conserved quantities around the profile functions $\varphi_c(x)$ and $\psi_c(x)$ of peakon solutions, under the case of time independence.

\begin{lemma}\label{lem1.5}
For $u, v \in H^{s}(\R)$, $s\geq 3$, if $\|u-\varphi_{c}\|_{H^{1}(\mathbb{R})} < \delta$ and $\|v-\psi_{c}\|_{H^{1}(\mathbb{R})} < \delta$ with $0<\delta<1/2$, then
\begin{equation*}
\left |E_{u}[u]-E_{u}[\varphi_{c}]\right |<2\sqrt{2}a\delta+\delta^2, \quad \left |E_{v}[v]-E_{u}[\psi_{c}]\right|<2\sqrt{2}b\delta+\delta^2
\end{equation*}
and
\begin{equation*}
\left|H[u,v]-H[\varphi_{c},\psi_{c}]\right|<2\sqrt{2}(a+b)\delta +6\delta^{2}, \quad \left|F[u,v]-F[\varphi_{c},\psi_{c}]\right|<C\delta+O(\delta^{2}),
\end{equation*}
where $C>0$ is a constant depending on $a$, $b$, $\|u\|_{H^s}$ and $\|v\|_{H^s}$.
\end{lemma}

\begin{proof}
Let $\tilde{u}=u-\varphi_{c}$ and $\tilde{v}=v-\psi_{c}$, for convenience. Since $\|\tilde{u}\|_{H^{1}(\mathbb{R})} < \delta$ and $\|\tilde{v}\|_{H^{1}(\mathbb{R})} < \delta$, it follows that
\begin{eqnarray*}
\begin{aligned}
\left|E_{u}[u]-E_{u}[\varphi_{c}]\right|=&\left|\|u_{0}\|^{2}_{H^{1}(\mathbb R)}-\|\varphi_{c}\|^{2}_{H^{1}(\mathbb R)}\right|\\
=&\left|\|u_{0}\|_{H^{1}(\mathbb R)}-\|\varphi_{c}\|_{H^{1}(\mathbb R)}\right|\left|\|u_{0}\|_{H^{1}(\mathbb R)}+\|\varphi_{c}\|_{H^{1}(\mathbb R)}\right|\\
\leq &\|\tilde{u}_{0}\|_{H^{1}(\mathbb{R})} \left(\|\tilde{u}_{0}\|_{H^{1}(\mathbb R)}+2\|\varphi_{c}\|_{H^{1}(\mathbb R)}\right)\leq 2\sqrt{2}\,a\delta+\delta^2.
\end{aligned}
\end{eqnarray*}
Similarly, we have
\begin{align*}
\left|E_{v}[v]-E_{v}[\psi_{c}]\right|<2\sqrt{2}\,b\delta+\delta^2.
\end{align*}

Next, we now estimate
\begin{eqnarray*}
\begin{aligned}
\Big|H[u, v]-&H[\varphi_{c},\psi_{c}]\Big|=\left|\int_{\mathbb R}\left(uv+u_xv_x\right)\,dx-\int_{\mathbb R}\left(\varphi_{c}\psi_{c}+\varphi'_{c}\psi'_{c}\right)dx\right|\\
=&\left|\int_{\mathbb R}\left(\tilde{u}\tilde{v}+u\tilde{v}+v\tilde{u}\right)\,dx+\int_{\mathbb R}\left[(u_x-\varphi'_{c})(v_x-\psi'_{c})+u_{x}(v_{x}-\psi'_{c})+v_{x}(u_{x}-\varphi'_{c})\right]\,dx\right|\\
\leq & \int_{\mathbb R}\left|\tilde{u}\tilde{v}\right|\,dx+\int_{\mathbb R}\left|\tilde{u}v\right|\,dx+\int_{\mathbb R}\left|(v-\varphi_{c})u\right|\,dx\\
&\quad +\int_{\mathbb R}\left|\tilde{u}'\tilde{v}'\right|dx+\int_{\mathbb R}\left|\tilde{u}'v_{x}\right|\,dx+\int_{\mathbb R}\left|\tilde{v}'u_{x}\right|\,dx.
\end{aligned}
\end{eqnarray*}
Using the H$\mathrm{\ddot{o}}$lder inequality
\begin{eqnarray*}
\int_{\mathbb R}\left|\tilde{u}v\right|\,dx \leq \left(\int_{\mathbb R}\tilde{u}^{2}\,dx\right)^{\frac12}\left(\int_{\mathbb R}v^2\,dx\right)^{\frac12}\leq \|\tilde{u}\|_{H^{1}(\mathbb{R})}\left(\|\tilde{v}\|_{H^{1}(\mathbb{R})}+\|\varphi_{c}\|_{H^{1}(\mathbb{R})}\right) \leq \delta^{2}+\sqrt{2}\,a \delta,
\end{eqnarray*}
we deduce that
\begin{align*}
\left|H[u, v]-H[\varphi_{c},\psi_{c}]\right|<2\sqrt{2}(a+b)\delta +6\delta^{2}.
\end{align*}

Finally, we estimate
\begin{eqnarray*}
\begin{aligned}
\Big|F&[u, v]-F[\varphi_{c},\psi_{c}]\Big|=\Bigg|\int_{\mathbb{R}}\left(u^2v^2+\frac{1}{3}u^2v_{x}^2+\frac{1}{3}v^2u_{x}^2+\frac{4}{3}uvu_{x}v_{x}-\frac{1}{3}u_{x}^2v_{x}^2\right)\,dx\\
&\qquad\qquad\qquad\qquad\quad\; -\int_{\mathbb{R}}\left[\varphi_{c}^{2}\psi_{c}^{2}+\frac{1}{3}\varphi_{c}^{2}(\psi'_{c})^{2}+\frac{1}{3}(\varphi'_{c})^{2}\psi_{c}^{2}+\frac{4}{3}\varphi_{c}\psi_{c}\varphi'_{c}\psi'_{c}-\frac{1}{3}(\varphi'_{c})^{2}(\psi'_{c})^{2}\right]\,dx\Bigg|\\
&\leq \int_{\mathbb{R}}\left|u^2v^2-\varphi_{c}^{2}\psi_{c}^{2}\right|\,dx+\frac{1}{3}\,\int_{\mathbb{R}}\left|u^2v_{x}^2-\varphi_{c}^{2}(\psi'_{c})^{2}\right|\,dx+\frac{1}{3}\,\int_{\mathbb{R}}\left|v^2u_{x}^2-(\varphi'_{c})^{2}\psi_{c}^{2}\right|\,dx\\
&\quad\quad+\frac{4}{3}\,\left|\int_{\mathbb{R}}\big(uvu_{x}v_{x}-\varphi_{c}\psi_{c}\varphi'_{c}\psi'_{c}\big)dx\right|+\frac{1}{3}\,\left|\int_{\mathbb{R}}\left[u_{x}^2v_{x}^2-(\varphi'_{c})^{2}(\psi'_{c})^{2}\right]dx\right|\\
&\triangleq I_{1}+\frac{1}{3}I_{2}+\frac{1}{3}I_{3}+\frac{4}{3}I_{4}+\frac{1}{3}I_{5}.
\end{aligned}
\end{eqnarray*}
For the first term $I_{1}$, we obtain
\begin{eqnarray*}
\begin{aligned}
I_{1}& \leq  \int_{\mathbb R}\left|\left(u^2-\varphi^{2}_{c}\right)v^2\right|\,dx+\int_{\mathbb R}\left|\left(v^2-\psi^{2}_{c}\right)\varphi^{2}_{c}\right|\,dx\\
&\leq  \|v\|^{2}_{L^\infty}\left(\int_{\mathbb R}\tilde{u}^2\,dx\right)^{\frac 12}\left(\int_{\mathbb R}(u+\varphi_{c})^2\,dx\right)^{\frac 12}+\|\varphi_{c}\|^{2}_{L^\infty}\left(\int_{\mathbb R}\tilde{v}^2\,dx\right)^{\frac 12}\left(\int_{\mathbb R}(v+\psi_{c})^2\,dx\right)^{\frac 12}\\
&\leq \|v\|^{2}_{L^\infty}\|\tilde{u}\|_{H^1}\left(\|\tilde{u}\|_{H^1}+2\|\varphi_{c}\|_{H^1}\right)+\|\varphi_{c}\|^{2}_{L^\infty}\|\tilde{v}\|_{H^1}\left(\|\tilde{v}\|_{H^1}+2\|\psi_{c}\|_{H^1}\right)\\
&\leq \frac{1}{2}\left(\delta+\sqrt{2}b\right)^2\delta\left(\delta+2\sqrt{2}a\right)+b^2\delta\left(\delta+2\sqrt{2}b\right)\leq 2\sqrt{2}b^2(a+b)\delta+O(\delta^2).
\end{aligned}
\end{eqnarray*}
For the second term $I_{2}$, we have
\begin{align*}
I_{2}&\leq \int_{\mathbb R}\left|\left(u^2-\varphi^{2}_{c}\right)v_{x}^2\right|\,dx+\int_{\mathbb R}\left|\left(v_{x}^2-\psi'^{2}_{c}\right)\varphi^{2}_{c}\right|\,dx\\
&\leq \|v_{x}\|^{2}_{L^\infty}\left(\int_{\mathbb R}(\tilde{u})^2\,dx\right)^{\frac 12}\left(\int_{\mathbb R}(u+\varphi_{c})^2\,dx\right)^{\frac 12}
+\|\varphi_{c}\|^{2}_{L^\infty}\left(\int_{\mathbb R}(\tilde{v}_{x})^2\,dx\right)^{\frac 12}\left(\int_{\mathbb R}(v_{x}+\psi'_{c})^2\,dx\right)^{\frac 12}\\
&\leq  \|v_{x}\|^{2}_{L^\infty}\|\tilde{u}\|_{H^1}\left(\|\tilde{u}\|_{H^1}+2\|\varphi_{c}\|_{H^1}\right)+\|\varphi_{c}\|^{2}_{L^\infty}\|\tilde{v}\|_{H^1}\left(\|\tilde{v}\|_{H^1}+2\|\psi_{c}\|_{H^1}\right)\\
&\leq \frac{1}{2}\left(\delta+\sqrt{2}b\right)^2\delta\left(\delta+2\sqrt{2}a\right)+b^2\delta\left(\delta+2\sqrt{2}b\right)\leq 2\sqrt{2}b^2(a+b)\delta+O(\delta^2).
\end{align*}
Similarly, for the term $I_{3}$, we get
\begin{align*}
I_{3}< 2\sqrt{2}a^2(a+b)\delta+O(\delta^{2}).
\end{align*}
For the fourth term $I_{4}$, we estimate
\begin{align*}
I_{4} & \leq \int_{\mathbb{R}}\left|\tilde{u}vu_{x}v_{x}\right|\,dx+\int_{\mathbb{R}}\left|\varphi_{c}\tilde{v}vu_{x}v_{x}\right|\,dx+\int_{\mathbb{R}}\left|\varphi_{c}\psi_{c}\tilde{u}_{x}v_{x}\right|\,dx+\int_{\mathbb{R}}\left|\varphi_{c}\psi_{c}\varphi'_{c}\tilde{v}_{x}\right|\,dx\\
& \leq \frac{1}{2}\|\tilde{u}\|_{L^{\infty}}\|v\|_{L^{\infty}}\int_{\mathbb{R}}(u^{2}_{x}+v^{2}_{x})\,dx+\frac{1}{2}\|\tilde{v}\|_{L^{\infty}}\|u\|_{L^{\infty}}\int_{\mathbb{R}}(u^{2}_{x}+v^{2}_{x})\,dx\\
&\;\;\; +\|\varphi_{c}\|_{L^{\infty}}\|\psi_{c}\|_{L^{\infty}}\left(\int_{\mathbb{R}}\tilde{u}_{x}^{2}\,dx\right)^{\frac{1}{2}}\left(\int_{\mathbb{R}}v^{2}_{x}\,dx\right)^{\frac{1}{2}}+\|\varphi_{c}\|_{L^{\infty}}\|\psi_{c}\|_{L^{\infty}}\left(\int_{\mathbb{R}}\tilde{v}_{x}^{2}\,dx\right)^{\frac{1}{2}}\left(\int_{\mathbb{R}}u^{2}_{x}\,dx\right)^{\frac{1}{2}}\\
& \leq \frac{1}{2}\left(\|\tilde{u}\|_{H^{1}(\mathbb{R})}\|v\|_{H^{1}(\mathbb{R})}+\|\tilde{v}\|_{H^{1}(\mathbb{R})}\|u\|_{H^{1}(\mathbb{R})}\right)\left(\|v\|^{2}_{H^{1}(\mathbb{R})}+\|u\|^{2}_{H^{1}(\mathbb{R})}\right)\\
&\quad +\frac{1}{2}\|\varphi_{c}\|_{H^{1}(\mathbb{R})}\|\psi_{c}\|_{H^{1}(\mathbb{R})}\left(\|\tilde{u}\|_{H^{1}(\mathbb{R})}\|v\|_{H^{1}(\mathbb{R})}+\|\tilde{v}\|_{H^{1}(\mathbb{R})}\|u\|_{H^{1}(\mathbb{R})}\right)\\
& \leq \sqrt{2}(a+b)(a^{2}+b^{2}+ab)\delta+O(\delta^{2}).
\end{align*}
For the fifth term $I_{5}$, we have
\begin{align*}
I_{5} & \leq \left|\int_{\mathbb{R}}u_{x}^{2}\left(v_{x}^2-(\psi'_{c})^{2}\right)\,dx\right|+\left|\int_{\mathbb{R}}(\psi'_{c})^{2}(u^{2}_{x}-(\varphi'_{c})^{2})\,dx\right|\\
& \leq \|u\|^{2}_{L^{\infty}}\left(\int_{\mathbb{R}}\tilde{v}_{x}^{2}\,dx\right)^{\frac{1}{2}}\left(\int_{\mathbb{R}}(v_{x}+\psi'_{c})^{2}\,dx\right)^{\frac{1}{2}}+\|\psi'_{c}\|^{2}_{L^{\infty}}\left(\int_{\mathbb{R}}\tilde{u}_{x}^{2}dx\right)^{\frac{1}{2}}\left(\int_{\mathbb{R}}(u_{x}+\varphi'_{c})^{2}dx\right)^{\frac{1}{2}}\\
& \leq \frac{1}{2}\left[\|u\|^{2}_{H^{1}(\mathbb{R})}\|\tilde{v}\|_{H^{1}(\mathbb{R})}\left(\|\tilde{v}\|_{H^{1}(\mathbb{R})}+2\|\psi_{c}\|_{H^{1}(\mathbb{R})}\right)+\|\psi'_{c}\|^{2}_{H^{1}(\mathbb{R})}\|\tilde{u}\|_{H^{1}(\mathbb{R})}\left(\|\tilde{u}\|_{H^{1}(\mathbb{R})}+2\|\varphi_{c}\|_{H^{1}(\mathbb{R})}\right)\right]\\
& \leq 2\sqrt{2}ab(a+b)\delta+O(\delta^{2}).
\end{align*}
Accordingly, for $0<\delta<1/2$, we have
\begin{align*}
\left|F(u,v)-F(\varphi_{c},\psi_{c})\right|\leq C\delta +O(\delta^{2}),
\end{align*}
where $C$ is a constant depending on $a$, $b$, $\|u\|_{H^s}$ and $\|v\|_{H^s}$, which completes the proof of this lemma.
\end{proof}

Now, we are in the position to prove that the strong solution satisfies the novel error estimates at some kind of critical point under the assumption of small perturbation of initial data around the profile of peakon solutions.

\begin{lemma}
Assume that $(u(t, x),\,v(t, x))$, $t\in [0,\,T)$, is the corresponding strong solution of the Cauchy problem \eqref{tcNK-1} with initial data $(u_0(x),\,v_0(x))$ satisfying $0\not\equiv u_0,\,v_0\in H^s(\R)$, $s\geq 3$ and $m_0=u_0-u_{0xx}\geq 0$, $n_0=v_0-v_{0xx}\geq 0$. If $(u_0(x),\,v_0(x))$ satisfies
\begin{equation*}
\|u_0-\varphi_{c}\|_{H^{1}(\mathbb{R})} < \delta \quad \mathrm{and} \quad \|v_0-\psi_{c}\|_{H^{1}(\mathbb{R})} < \delta,
\end{equation*}
with $0<\delta<1/2$, then there exists a constant $C>0$ such that
\begin{eqnarray*}
|u(t, \xi(t))-a|<C\delta^{\frac{1}{2}} \quad \mathrm{and} \quad |v(t, \xi(t))-b|<C\delta^{\frac{1}{2}}, \quad \forall t\in [0,\,T),
\end{eqnarray*}
where $\xi(t)\in \R$ is located such that $u(t, \xi(t))v(t, \xi(t))=\max_{x\in \mathbb R}\{u(t, x)v(t, x)\}=M(t)$.
\end{lemma}

\begin{proof}
In view of the conservation of the functionals $H[u, v]$ and $F[u, v]$, as well as the invariant property established in Lemma \ref{nonnegative} of the strong solution $(u(t, x),\,v(t, x))$, we deduce from Lemma 3.5 that for any $t\in [0,\,T)$,
\begin{equation}\label{H_t}
\Big|H[u(t, \cdot), v(t, \cdot)]-H[\varphi_{c},\psi_{c}]\Big|<2\sqrt{2}(a+b)\delta +6\delta^{2}
\end{equation}
and
\begin{equation}\label{F_t}
\Big|F[u(t, \cdot), v(t, \cdot)]-F[\varphi_{c},\psi_{c}]\Big|<C\delta+O(\delta^{2}),
\end{equation}
where the constants involved depend on $a$, $b$, $\|u_0\|_{H^s}$ and $\|v_0\|_{H^s}$.

Furthermore, we have obtained in Lemma \ref{lem1.4} the following dynamical inequality
\begin{align*}
F[u, v]-\frac{4}{3}M(t)H[u, v]+\frac{4}{3}M(t)^2 \leq 0.
\end{align*}
Then, define the polynomial $P(y)$ to be
\begin{align*}
P(y)=\frac{4}{3}y^{2}-\frac{4}{3}H[u, v]y+F[u, v].
\end{align*}
For the peakons, we have $H[\varphi_{c}, \psi_{c}]=2ab$ and $F[\varphi_{c}, \psi_{c}]=4a^{2}b^{2}/3$, thus the above polynomial with respect to peakons is defined by
\begin{align*}
P_{0}(y)=\frac{4}{3}y^{2}-\frac{4}{3}H[\varphi_{c}, \psi_{c}]y+F[\varphi_{c}, \psi_{c}]=\frac{4}{3}(y-ab)^{2}.
\end{align*}
Note that
\begin{align*}
P(y)=P_0(y)-\frac{4}{3}\big(H[u, v]-H[\varphi_{c}, \psi_{c}]\big)\,y+\big(F[u, v]-F[\varphi_{c}, \psi_{c}]\big).
\end{align*}
It follows that for any $t\in [0,\,T)$
\begin{align*}
\frac{4}{3}\big(M(t)-ab\big)^{2}=P_{0}(M(t))\leq \frac{4}{3}\big|H(u, v)-H(\varphi_{c}, \psi_{c})\big|M(t)+\big|F(u, v)-F(\varphi_{c}, \psi_{c})\big|,
\end{align*}
which together with \eqref{H_t} and \eqref{F_t} leads to
\begin{align}\label{M-ab}
\big|M(t)-ab\big|\leq C\delta^{\frac{1}{2}}.
\end{align}
On the other hand, since
\begin{align*}
2u^{2}(t, \xi(t))\leq \|u_0\|^{2}_{H^{1}(\mathbb{R})}, \quad  2v^{2}(t, \xi(t))\leq \|v_0\|^{2}_{H^{1}(\mathbb{R})},
\end{align*}
we have
\begin{align*}
0<u(t, \xi(t))\leq a+\tilde{C}\delta^{\frac{1}{2}}, \quad  0<v(t, \xi(t))\leq b+\tilde{C}\delta^{\frac{1}{2}}.
\end{align*}
It follows from \eqref{M-ab} and the above estimate that
\begin{align*}
ab-\tilde{C}\delta^{\frac 12}\leq u(\xi)v(\xi)\leq (a+\tilde{C}\delta^{\frac 12})v(\xi),
\end{align*}
which implies
\begin{align*}
v(t, \xi(t))\geq b-C\delta^{\frac{1}{2}}.
\end{align*}
Similarly, we have
\begin{align*}
u(t, \xi(t))\geq a-C\delta^{\frac{1}{2}}.
\end{align*}
This finishes the proof of the lemma.
\end{proof}

\begin{proof}[Proof of Theorem \ref{thm1.1}]
Let $u \in C([0,T); H^{s})$ and $v \in C([0, T); H^{s})$, $s \geq 3$, be the unique strong solution for Cauchy problem \eqref{tcNK-1} of the two-component Novikov system \eqref{tcNK} with initial data $u(0, x)=u_{0}(x)$ and $v(0, x)=v_{0}(x)$. Since $E_u[u]$ and $E_v[v]$ are both conserved, we deduce that
\begin{align}\label{conserved}
E_u[u(t, \cdot)]=E_u[u_0] \quad and \quad E_v[v(t, \cdot)]=E_v[v_0], \quad \forall t\in (0,\,T).
\end{align}

By \eqref{conserved} and Lemma \ref{lem1.1}, we infer that for any $t\in [0,\,T)$
\begin{align*}
\|u(t, \cdot)-\varphi_{c}(\cdot-\xi(t))\|^{2}_{H^{1}(\mathbb{R})}=&E_{u}[u_{0}]-E_{u}[\varphi_{c}]-4a(u(t, \xi(t))-a),\\
\|v(t, \cdot)-\psi_{c}(\cdot-\xi(t))\|^{2}_{H^{1}(\mathbb{R})}=&E_{v}[v_{0}]-E_{v}[\psi_{c}]-4b(v(t, \xi(t))-b).
\end{align*}
Combining the above two dynamical pointwise identities with Lemmas 3.5 and 3.6 , we conclude that for any $t\in [0,\,T)$
\begin{align*}
\|u-\varphi_{c}(\cdot-\xi(t))\|_{H^{1}(\mathbb{R})}<C\delta^{\frac{1}{4}} \quad \mathrm{and} \quad \|v-\psi_{c}(\cdot-\xi(t))\|_{H^{1}(\mathbb{R})}<C\delta^{\frac{1}{4}},
\end{align*}
where as in Lemma 3.6, $C$ is a constant depending on $a$, $b$, $\|u\|_{H^s}$ and $\|v\|_{H^s}$. This thus completes the proof of Theorem \ref{thm1.1}.
\end{proof}

\section{Stability of train-profiles}

In this section, we are devoted to proving orbital stability of train-profiles of two-component peakons $(\varphi_c, \, \psi_c)$. For $\alpha>0$ and $L>0$, we define the following neighborhood of all the sums of N peakons of speed $c_1,...,c_N$ with spatial shifts $x_i$ that satisfy $x_i-x_{i-1} \geq L$,
\begin{align*}
U( \alpha , L) = \Big\{(u , v) & \in H^{1}(\mathbb{R}) \times H^{1}(\mathbb{R}), \\
& \inf_{x_i-x_{i-1} \geq L}{{\Big\| u- \sum_{i=1}^N{\varphi_{c_i}(\cdot - {x}_i)} \Big\|_{H^1}}+{\Big\| v- \sum_{i=1}^N{\psi_{c_i}(\cdot - {x}_i)} \Big\|_{H^1}}} \leq \alpha \Big\}.
\end{align*}

By the continuity of the map $t \mapsto (u(t),v(t))$ from $[0,T[$ into $H^1(\mathbb{R})\times H^1(\mathbb{R})$, to prove Theorem 1.2 it suffices to verify that there exist $ A>0$, $\epsilon_0>0$ and $L_0>0$ such that for any $L>L_0$ and $0<\epsilon<\epsilon_0$, if $(u_{0},v_{0})$ satisfies \eqref{initialdata-1} and for some $0<t_0<T$, we have
\begin{align}\label{uvinitial}
(u(t),v(t)) \in U\LC A({\epsilon}^{ \frac{1}{4}}+L^{- \frac{1}{8}}),\frac{L}{2}\RC \quad {\forall}t \in [0, t_0],
\end{align}
then
\begin{align}\label{uvending}
(u(t_{0}),v(t_{0})) \in U\LC \frac{A}{2}({\epsilon}^{ \frac{1}{4}}+L^{- \frac{1}{8}}),\frac{2L}{3}\RC.
\end{align}
Therefore, in the sequel of this section we will assume \eqref{uvinitial} for some $0< \epsilon < \epsilon_{0}$ and $L>L_{0}$, with $A$, $\epsilon_{0}$ and $L_{0}$ to be specified later.
\subsection{Control of the distance between the peakons}
In this subsection we shall prove that the different bumps of $u$ and $v$ that are individually close to their own peakons and get away from each others as time is increasing. This is crucial in our analysis since we do not know how to manage strong interactions.

\begin{prop}\label{implicitfunction}
 Assume $(u_{0}, v_{0})$ satisfies \eqref{initialdata-1}, and there exist $\alpha_{0} > 0$, $L_{0} > 0$ and $C_{0} > 0$ such that for all $0 < \alpha < \alpha_{0}$, $L >L_{0} >0$. If $(u(t),v(t)) \in U(\alpha, \frac {L}{2})$ on $[0, t_{0}]$ for some $0< t_{0} <T$, then there exist $C^{1} functions$ $\tilde{x}_{1} ,..., \tilde{x}_{N}$ defined on $[0, t_{0}]$ such that
\begin{align}\label{initial-3-1}
\frac{d}{dt}\tilde{x}_{i}=c_{i}+O(\sqrt{\alpha})+O(L^{-1}),\;\; i=1,...,N,
\end{align}
\begin{align}\label{initial-3-2}
{\Big\| u(t)- \sum_{i=1}^N{\varphi_{c_i}(\cdot - \tilde{x}_i(t))} \Big\|_{H^1}}+{\Big\| v(t)- \sum_{i=1}^N{\psi_{c_i}(\cdot - \tilde{x}_i(t))} \Big\|_{H^1}}=O(\sqrt{\alpha}),
\end{align}
\begin{align}\label{initial-3-3}
\tilde{x}_{i}(t)-\tilde{x}_{i-1}(t) \geq \frac{3L}{4} + \frac{c_{i}-c_{i-1}}{2}t, \;\;i=2,...,N.
\end{align}
Moreover, setting $J_{i}:=[y_{i}(t),y_{i-1}(t)]$,\;\; $i=1,...,N$, with
\begin{align}\label{initial-3-4}
y_{1}=-\infty,\quad y_{N+1}=+\infty \quad and \quad y_{i}(t)=\frac{\tilde{x}_{i-1}(t)+\tilde{x}_{i}(t)}{2},\;\; i=2,...,N,
\end{align}
it holds
\begin{align}\label{initial-3-5}
|\tilde{x}_{i}(t)-x_{i}(t)| \leq \frac{L}{12} ,\;\; i=1,...,N,
\end{align}
where $x_{1}(t),...,x_{N}(t)$ are any point such that
\begin{align}\label{initial-3-6}
u(t,x_{i}(t))v(t,x_{i}(t))= \max_{x \in J_{i}(t)}u(t)v(t),\;\; i=1,...,N.
\end{align}
\end{prop}

To prove this proposition we use a modulation argument. The strategy is to construct $C^{1}$ functions $\tilde{x}_{1} ,..., \tilde{x}_{N}$ on $[0,t_{0}]$ satisfying a suitable orthogonality condition, see \eqref{orthogonality}. Thanks to this orthogonality condition, we will be able to prove that the speed of the $\tilde{x}_{i}$ stays close to $c_{i}$ on $[0,t_{0}]$.

\begin{proof}
For $Z=(z_{1},...,z_{N}) \in \mathbb{R}^{N}$ satisfying $z_{i}-z_{i-1}>L/2$, we set
\begin{align}\label{sumofsolitons}
R_{Z}(\cdot)=\sum_{i=1}^{N}\varphi_{c_{i}}(\cdot - z_{i}) \quad and \quad S_{Z}(\cdot)=\sum_{i=1}^{N}\psi_{c_{i}}(\cdot - z_{i}).
\end{align}
For $\alpha_{0} > 0$, $L_{0} > 0$, we define the function
\begin{align*}
Y:(-\alpha ,\alpha )^{N} \times B_{H^{1} \times H^{1}}((R_{Z},S_{Z}),\alpha) \mapsto {\mathbb{R}}^{N}, \\
(y_{1},...,y_{N},u,v) \mapsto (Y^{1}(y_{1},...,y_{N},u,v),...,Y^{N}(y_{1},...,y_{N},u,v))
\end{align*}
with
\begin{align*}
Y^{i}(y_{1},...,y_{N},u,v)= &\int_{\mathbb{R}}\bigg{(}(u-\sum_{i=1}^{N}\varphi_{c_{j}}(\cdot - z_{j} -y_{j})) \partial_{x}\varphi_{c_{i}}(\cdot - z_{i} -y_{i}) \\
&\qquad \quad+ (v-\sum_{i=1}^{N}\psi_{c_{j}}(\cdot - z_{j} -y_{j})) \partial_{x}\psi_{c_{i}}(\cdot - z_{i} -y_{i})\bigg{)},
\end{align*}
$Y$ is clearly of class $C^{1}$. For $i=1,...,N,$
\begin{align*}
\frac{\partial{Y}^{i}}{\partial y_{i}}(y_{1},...,y_{N},u,v)=&\int_{\mathbb{R}}\bigg{(}(u_{x}-\sum_{i \ne j}^{N}\partial_{x} \varphi_{c_{j}}(\cdot - z_{j} -y_{j})) \partial_{x}\varphi_{c_{i}}(\cdot - z_{i} -y_{i}) \\
&\qquad\qquad + (v_{x}-\sum_{i \ne j}^{N}\partial_{x} \psi_{c_{j}}(\cdot - z_{j} -y_{j})) \partial_{x}\psi_{c_{i}}(\cdot - z_{i} -y_{i})\bigg{)},
\end{align*}
and $\forall j\ne i,$
\begin{align*}
\frac{\partial{Y}^{i}}{\partial y_{j}}(y_{1},...,y_{N},u,v)=&\int_{\mathbb{R}}\bigg{(}\partial_{x} \varphi_{c_{j}}(\cdot - z_{j} -y_{j}) \partial_{x}\varphi_{c_{i}}(\cdot - z_{i} -y_{i}) \\
&\qquad\qquad + \partial_{x} \psi_{c_{j}}(\cdot - z_{j} -y_{j}) \partial_{x}\psi_{c_{i}}(\cdot - z_{i} -y_{i})\bigg{)}.
\end{align*}
Hence
\begin{align*}
\frac{\partial{Y}^{i}}{\partial y_{i}}(0,...,0,R_{Z},S_{Z})=\| \partial_{x}\varphi_{c_{i}} \|_{L^2}^2+\| \partial_{x}\psi_{c_{i}} \|_{L^2}^2 \geq a_{1}^2+b_{1}^{2},
\end{align*}
and, for $\forall j\ne i$, using the exponential decay of $\varphi_{c}$ and that $z_{i}-z_{i-1} >L $ we infer  for $L_{0}$ large enough that (recall that $L>L_{0}$),
\begin{equation*}
\begin{aligned}
\frac{\partial{Y}^{i}}{\partial y_{j}}(0,...,0,R_{Z},S_{Z})=&\int_{\mathbb{R}}\bigg{(}\partial_{x} \varphi_{c_{j}}(\cdot - z_{j} ) \partial_{x}\varphi_{c_{i}}(\cdot - z_{i} ) + \partial_{x} \psi_{c_{j}}(\cdot - z_{j}) \partial_{x}\psi_{c_{i}}(\cdot - z_{i})\bigg{)}\\
&\leq O(e^{-\frac{L}{4}}).
\end{aligned}
\end{equation*}
We conclude that, for $L>0$ large enough, $D_{(y_{1},...,y_{N})}Y(0,...,0,R_{Z},S_{Z})=D+P$ where $D$ is an invertible diagonal matrix with $\|D^{-1}\|\leq (a^{2}_{1}+b^{2}_{1})^{-n}$ and $\|P\|\leq O(e^{-L/4})$. Hence there exists $L_{0}>0$ such that for $L>L_{0}$, $D_{(y_{1},...,y_{N})}Y(0,...,0,R_{Z},S_{Z})$ is invertible with an inverse matrix of norm smaller than $2(a^{2}_{1}+b^{2}_{1})^{-n}$. The implicit function theorem implies that there exists $\beta_{0}>0$ and $C^{1}$ functions $y_{1},y_{2},...,y_{N}$ from $B_{H^1 \times H^1}((R_{Z},S_{Z}),\beta_{0})$ to a neighborhood of $(0,0,...,0)$ which uniquely determined such that
\begin{align*}
Y(y_{1},...,y_{N},u,v)=0 \quad {\rm for}\;\;  {\rm  all} \quad (u,v) \in B((R_{Z},S_{Z}),\beta_{0}).
\end{align*}
In particular, there exists $C_{0}>0$ such that if $(u,v) \in B((R_{Z},S_{Z}),\beta)$, with $0<\beta \leq \beta_{0}$, then
\begin{align}\label{implicitfunc}
\sum^{N}_{i=1}{\big{|}y_{i}(u,v)\big{|}}\leq C_{0}\beta.
\end{align}
Note that $\beta_{0}$ and $C_{0}$ depend on only $a_{1}$, $b_{1}$ and $L_{0}$ and not on the point $(z_{1},...,z_{N})$. For $(u,v) \in B((R_{Z},S_{Z}),\beta_{0})$ we set $\tilde{x}_{i}(u,v)=z_{i}+y_{i}(u,v)$. Assuming that $\beta_{0}\leq L_{0}/(8C_{0})$, $\tilde{x}_{1},...,\tilde{x}_{N}$ are thus $C^{1}$ functions on $B((R_{Z},S_{Z}),\beta)$ satisfying
\begin{align}\label{dist-1}
\tilde{x}_{j}(u,v)-\tilde{x}_{j-1}(u,v)>\frac{L}{2}-2C_{0}\beta>\frac{L}{4}.
\end{align}
For $L>L_{0}$ and $0<\alpha<\alpha_{0}<\beta_{0}/2$ to be chosen later, we define the modulation of $(u,v) \in U(\alpha,L/2)$ in the following way, the trajectory of $(u,v)$ is covered by a finite number of open balls:
\begin{align*}
\Big{\{}\big{(}u(t),v(t)\big{)}, t \in [0,t_{0}]\Big{\}} \subset \bigcup_{k=1,...,M}B\big{(}(R_{Z^{k}},S_{Z^{k}}),2\alpha\big{)}.
\end{align*}
It is worth noticing that, since $0<\alpha<\alpha_{0}<\beta_{0}/2$, the functions $\tilde{x}_{i}(u,v)$ are uniquely determined for $(u,v)\in B((R_{Z^{k}},S_{Z^{k}}),2\alpha)\bigcap B((R_{Z^{k'}},S_{Z^{k'}}),2\alpha)$. We can thus define the functions $t\mapsto \tilde{x}_{i}(t)$ on $[0,t_{0}]$ by setting $\tilde{x}_{i}(t)=\tilde{x}_{i}(u(t),v(t))$. By construction
\begin{eqnarray}\label{orthogonality}
\begin{aligned}
\int_{\mathbb{R}} &\Big{(}\Big{(}u(t,\cdot)-\sum^{N}_{j=1}\varphi_{c_{j}}(\cdot-\tilde{x}_{j}(t))\Big{)}\partial_{x}\varphi_{c_{i}}(\cdot-\tilde{x}_{i}(t))\\
&\qquad \qquad\qquad +\Big{(}v(t,\cdot)-\sum^{N}_{j=1}\psi_{c_{j}}(\cdot-\tilde{x}_{j}(t))\Big{)}\partial_{x}\psi_{c_{i}}(\cdot-\tilde{x}_{i}(t))\Big{)}dx=0.
\end{aligned}
\end{eqnarray}
Moreover, on account of \eqref{implicitfunc} and the fact that $\varphi''_{c}$ and $\psi''_{c}$ are the sum of an $L^{1}$ function and a Dirac mass, we claim
\begin{align}\label{preestimate}
\Big{\|}\Big{(}u(t),v(t)\Big{)}-\Big{(}R_{\tilde{X}(t)},S_{\tilde{X}(t)}\Big{)}\Big{\|}_{H^{1}\times H^{1}}\leq O(\sqrt{\alpha}), \quad \forall t \in [0,t_{0}].
\end{align}
Indeed, one can calculate
\begin{align*}
&\Big{\|}\Big{(}u(t),v(t)\Big{)}-\Big{(}R_{\tilde{X}(t)},S_{\tilde{X}(t)}\Big{)}\Big{\|}_{H^{1}\times H^{1}} \triangleq \Big{\|}u(t)-R_{\tilde{X}(t)}\Big{\|}_{H^{1}}+\Big{\|}v(t)-S_{\tilde{X}(t)}\Big{\|}_{H^{1}}\\
&\leq \Big{\|}u(t)-R_{{Z^{k}}(t)}\Big{\|}_{H^{1}}+\Big{\|}R_{{Z^{k}}(t)}-R_{\tilde{X}(t)}\Big{\|}_{H^{1}}+\Big{\|}v(t)-S_{{Z^{k}}(t)}\Big{\|}_{H^{1}}+\Big{\|}S_{{Z^{k}}(t)}-S_{\tilde{X}(t)}\Big{\|}_{H^{1}}\\
&\leq
\alpha+\hat{C}\sum^{N}_{i=1}\Big{\|}\varphi(\cdot-z^{k}_{i})-\varphi(\cdot-z^{k}_{i}-y_{i}(u,v))\Big{\|}_{H^{1}}\\
&\leq \alpha+\hat{C}\sum^{N}_{i=1}\Big{(}2(1-e^{-y_{i}})^{2}+2y_{i}+O(y^{2}_{i})\Big{)}^{\frac{1}{2}}\leq O(\sqrt{\alpha}).
\end{align*}
Let us now prove that the speed of $\tilde{x}_{i}$ stays close to $c_{i}$. We set
\begin{align*}
&R_{j}(t)=\varphi_{c_{j}}(\cdot-\tilde{x}_{j}(t)) \quad \mathrm{and} \quad \tilde{u}(t)=u(t)-\sum^{N}_{j=1}R_{j}(t)=u(t,\cdot)-R_{\tilde{X}(t)},\\
&S_{j}(t)=\psi_{c_{j}}(\cdot-\tilde{x}_{j}(t)) \quad \mathrm{and} \quad \tilde{v}(t)=v(t)-\sum^{N}_{j=1}S_{j}(t)=u(t,\cdot)-S_{\tilde{X}(t)}.
\end{align*}
Differentiating \eqref{orthogonality} with respect to $t$ we get
\begin{align*}
\int_{\mathbb{R}}\Big{(}\tilde{u}_{t}\partial_{x}R_{i}+\tilde{v}_{t}\partial_{x}S_{i}\Big{)}=\dot{\tilde{x}}_{i}\big{(}<\partial^{2}_{x}R_{i},\tilde{u}>_{H^{-1},H^{1}}+<\partial^{2}_{x}S_{i},\tilde{v}>_{H^{-1},H^{1}}\big{)},
\end{align*}
and thus
\begin{align}\label{speedestimate}
\bigg{|}\int_{\mathbb{R}}\Big{(}\tilde{u}_{t}\partial_{x}R_{i}+\tilde{v}_{t}\partial_{x}S_{i}\Big{)}\bigg{|}\leq |\dot{\tilde{x}}_{i}-c_{i}|\Big(O\big{(}\|\tilde{u}\|_{H^{1}}\big{)}+O\big{(}\|\tilde{v}\|_{H^{1}}\big{)}\Big)+\Big(O\big{(}\|\tilde{u}\|_{H^{1}}\big{)}+O\big{(}\|\tilde{v}\|_{H^{1}}\big{)}\Big).
\end{align}
Substituting $u$ by $\tilde{u}+\sum^{N}_{j=1}R_{j}(t)$ and $v$ by $\tilde{v}+\sum^{N}_{j=1}S_{j}(t)$ into  \eqref{weakform} and using
\begin{equation}
\left\{
\begin{aligned}
&\partial_{t}R_{i}+\big{(}\dot{\tilde{x}}_{i}(t)-c_{i}\big{)}\partial_{x}R_{i}+R_{i}S_{i}\partial_{x}R_{i}+P_{x}\ast\Big{(}\frac{1}{2}(\partial_{x}R_{i})^{2}S_{i}+R_{i}(\partial_{x}R_{i})(\partial_{x}S_{i})+R^{2}_{i}S_{i}\Big{)}\\
&+P\ast\Big{(}\frac{1}{2}(\partial_{x}R_{i})^{2}\partial_{x}S_{i}\Big{)}=0,\\
&\partial_{t}S_{i}+\big{(}\dot{\tilde{x}}_{i}(t)-c_{i}\big{)}\partial_{x}S_{i}+S_{i}R_{i}\partial_{x}S_{i}+P_{x}\ast\Big{(}\frac{1}{2}(\partial_{x}S_{i})^{2}R_{i}+S_{i}(\partial_{x}S_{i})(\partial_{x}R_{i})+S^{2}_{i}R_{i}\Big{)}\\
&+P\ast\Big{(}\frac{1}{2}(\partial_{x}S_{i})^{2}\partial_{x}R_{i}\Big{)}=0,
\end{aligned}
\right.
\end{equation}
we infer that $(\tilde{u},\tilde{v})$ satisfies on $[0,t_{0}]$ that
\begin{equation*}\label{speedequation}
\begin{aligned}
&\tilde{u}_{t}-\sum^{N}_{i=1}(\dot{\tilde{x}}_{i}(t)-c_{i})\partial_{x}R_{i}+\big{(}\tilde{u}+\sum^{N}_{j=1}R_{j}\big{)}\big{(}\tilde{v}+\sum^{N}_{j=1}S_{j}\big{)}\Big{(}\tilde{u}_{x}+\sum^{N}_{j=1}R_{jx}\Big{)}-\sum^{N}_{j=1}(R_{j}S_{j}\partial_{x}R_{j})\\
&\quad +P_{x}\ast\bigg{(}\frac{1}{2}\Big{(}\tilde{u}_{x}+\sum^{N}_{j=1}R_{jx}\Big{)}^{2}(\tilde{v}+\sum^{N}_{j=1}S_{j})
+(\tilde{u}+\sum^{N}_{j=1}R_{j})\Big{(}\tilde{u}_{x}+\sum^{N}_{j=1}R_{jx}\Big{)}\Big{(}\tilde{v}_{x}+\sum^{N}_{j=1}S_{jx}\Big{)}\\
&\qquad\quad+(\tilde{u}+\sum^{N}_{j=1}R_{j})^{2}(\tilde{v}+\sum^{N}_{j=1}S_{j})-\frac{1}{2}\sum^{N}_{j=1}R^{2}_{jx}S_{j}-\sum^{N}_{j=1}R_{j}R_{jx}S_{jx}-\sum^{N}_{j=1}R^{2}_{j}S_{j}\bigg{)}\\
&\qquad\qquad+P\ast \bigg{(}\frac{1}{2}\big{(}\tilde{u}_{x}+\sum^{N}_{j=1}R_{jx}\big{)}^{2}\big{(}\tilde{v}_{x}+\sum^{N}_{j=1}S_{jx}\big{)}-\frac{1}{2}\sum^{N}_{j=1}R^{2}_{jx}S_{jx}\bigg{)}=0,\\
\end{aligned}
\end{equation*}
\begin{eqnarray}\label{speedequation}
\begin{aligned}
&\tilde{v}_{t}-\sum^{N}_{i=1}(\dot{\tilde{x}}_{i}(t)-c_{i})\partial_{x}S_{i}+\big{(}\tilde{v}+\sum^{N}_{j=1}S_{j}\big{)}\big{(}\tilde{u}+\sum^{N}_{j=1}R_{j}\big{)}\Big{(}\partial_{x}\big{(}\tilde{v}+\sum^{N}_{j=1}S_{j}\big{)}\Big{)}-\sum^{N}_{j=1}(S_{j}R_{j}\partial_{x}S_{j})\\
&\quad+P_{x}\ast\bigg{(}\frac{1}{2}\Big{(}\tilde{v}_{x}+\sum^{N}_{j=1}S_{jx}\Big{)}^{2}(\tilde{u}+\sum^{N}_{j=1}R_{j})
+(\tilde{v}+\sum^{N}_{j=1}S_{j})\Big{(}\tilde{v}_{x}+\sum^{N}_{j=1}S_{jx}\Big{)}\Big{(}\tilde{u}_{x}+\sum^{N}_{j=1}R_{jx}\Big{)}\\
&\qquad+(\tilde{v}+\sum^{N}_{j=1}S_{j})^{2}(\tilde{u}+\sum^{N}_{j=1}R_{j})-\frac{1}{2}\sum^{N}_{j=1}S^{2}_{jx}R_{j}-\sum^{N}_{j=1}S_{j}S_{jx}R_{jx}-\sum^{N}_{j=1}S^{2}_{j}R_{j}\bigg{)}\\
&\qquad\qquad+P\ast \bigg{(}\frac{1}{2}\big{(}\tilde{v}_{x}+\sum^{N}_{j=1}S_{jx}\big{)}^{2}\big{(}\tilde{u}_{x}+\sum^{N}_{j=1}R_{jx}\big{)}-\frac{1}{2}\sum^{N}_{j=1}S^{2}_{jx}R_{jx}\bigg{)}=0.
\end{aligned}
\end{eqnarray}
Taking the $L^{2}$ scalar product with $\partial_{x}R_{i}$ in the first equation of \eqref{speedequation} and $\partial_{x}S_{i}$ in the second equation of \eqref{speedequation}, summing up the resulting equations, integrating by parts, and using the decay of $R_{j}$, $S_{j}$ and their first derivative, we claim
\begin{align}\label{speedfinestimate}
|\dot{\tilde{x}}_{i}(t)-c_{i}|\leq O(\sqrt{\alpha})+O(e^{-\frac{L}{8}}).
\end{align}
Indeed, according to the above, we obtain that
\begin{align*}
&(\dot{\tilde{x}}_{i}-c_{i})\big{(}<\partial^{2}_{x}R_{i},\tilde{u}>_{H^{-1},H^{1}}+<\partial^{2}_{x}S_{i},\tilde{v}>_{H^{-1},H^{1}}\big{)}\\
&+c_{i}\big{(}<\partial^{2}_{x}R_{i},\tilde{u}>_{H^{-1},H^{1}}+<\partial^{2}_{x}S_{i},\tilde{v}>_{H^{-1},H^{1}}\big{)}-(\dot{\tilde{x}}_{i}-c_{i})\int_{\mathbb{R}}\bigg{(}(\partial_{x}R_{i})^{2}+(\partial_{x}S_{i})^{2}\bigg{)}dx\\
&=\sum_{i\ne j}(\dot{\tilde{x}}_{j}-c_{j})\Big{(}\int_{\mathbb{R}}\big{(}\partial_{x}R_{i}\partial_{x}R_{j}+\partial_{x}S_{i}\partial_{x}S_{j}\big{)}dx\Big{)}\\
&\quad -\int_{\mathbb{R}}\Bigg{(}\bigg{(}\big{(}\tilde{u}+\sum^{N}_{j=1}R_{j}\big{)}\big{(}\tilde{v}+\sum^{N}_{j=1}S_{j}\big{)}\Big{(}\tilde{u}_{x}+\sum^{N}_{j=1}R_{jx}\Big{)}-\sum^{N}_{j=1}(R_{j}S_{j}\partial_{x}R_{j})\bigg{)}\partial_{x}R_{i}\Bigg{)}dx\\
&\quad -\int_{\mathbb{R}}\Bigg{(}\bigg{(}\big{(}\tilde{v}+\sum^{N}_{j=1}S_{j}\big{)}\big{(}\tilde{u}+\sum^{N}_{j=1}R_{j}\big{)}\Big{(}\partial_{x}\big{(}\tilde{v}+\sum^{N}_{j=1}S_{j}\big{)}\Big{)}-\sum^{N}_{j=1}(S_{j}R_{j}\partial_{x}S_{j})\bigg{)}\partial_{x}S_{i}\Bigg{)}dx\\
&\quad +\int_{\mathbb{R}}\Bigg{(}P\ast\bigg{(}\frac{1}{2}\Big{(}\tilde{u}_{x}+\sum^{N}_{j=1}R_{jx}\Big{)}^2(\tilde{v}+\sum^{N}_{j=1}S_{j})
+(\tilde{u}+\sum^{N}_{j=1}R_{j})\Big{(}\tilde{u}_{x}+\sum^{N}_{j=1}R_{jx}\Big{)}\Big{(}\tilde{v}_{x}+\sum^{N}_{j=1}S_{jx}\Big{)}\\
&\quad +(\tilde{u}+\sum^{N}_{j=1}R_{j})^{2}(\tilde{v}+\sum^{N}_{j=1}S_{j})-\frac{1}{2}\sum^{N}_{j=1}R^{2}_{jx}S_{j}-\sum^{N}_{j=1}R_{j}R_{jx}S_{jx}-\sum^{N}_{j=1}R^{2}_{j}S_{j}\bigg{)}\partial^{2}_{x}R_{i}\Bigg{)}dx\\
&\quad +\int_{\mathbb{R}}\Bigg{(}P\ast\bigg{(}\frac{1}{2}\Big{(}\tilde{v}_{x}+\sum^{N}_{j=1}S_{jx}\Big{)}^2(\tilde{u}+\sum^{N}_{j=1}R_{j})
+(\tilde{v}+\sum^{N}_{j=1}S_{j})\Big{(}\tilde{v}_{x}+\sum^{N}_{j=1}S_{jx}\Big{)}\Big{(}\tilde{u}_{x}+\sum^{N}_{j=1}R_{jx}\Big{)}\\
&\quad +(\tilde{v}+\sum^{N}_{j=1}S_{j})^{2}(\tilde{u}+\sum^{N}_{j=1}R_{j})-\frac{1}{2}\sum^{N}_{j=1}S^{2}_{jx}R_{j}-\sum^{N}_{j=1}S_{j}S_{jx}R_{jx}-\sum^{N}_{j=1}S^{2}_{j}R_{j}\bigg{)}\partial^{2}_{x}S_{i}\Bigg{)}dx\\
&\quad -\int_{\mathbb{R}}\Bigg{(}P\ast\bigg{(}\frac{1}{2}\big{(}\tilde{u}_{x}+\sum^{N}_{j=1}R_{jx}\big{)}^{2}\big{(}\tilde{v}_{x}+\sum^{N}_{j=1}S_{jx}\big{)}-\frac{1}{2}\sum^{N}_{j=1}R^{2}_{jx}S_{jx}\bigg{)}\partial_{x}R_{i}\Bigg{)}dx\\
&\quad -\int_{\mathbb{R}}\Bigg{(}P\ast\bigg{(}\frac{1}{2}\big{(}\tilde{v}_{x}+\sum^{N}_{j=1}S_{jx}\big{)}^{2}\big{(}\tilde{u}_{x}+\sum^{N}_{j=1}R_{jx}\big{)}-\frac{1}{2}\sum^{N}_{j=1}S^{2}_{jx}R_{jx}\bigg{)}\partial_{x}S_{i}\Bigg{)}dx.
\end{align*}
For every term, we have the following estimates
\begin{eqnarray*}
\begin{aligned}
<\partial^{2}_{x}R_{i},\tilde{u}>_{H^{-1},H^{1}}&=\int_{\mathbb{R}}\partial^{2}_{x}R_{i}\tilde{u}\,dx=\int_{\mathbb{R}}\left[R_{i}-2a_i\delta\left(x-\tilde{x}_{i}(t)\right))\right]\tilde{u}\,dx\\
&=\int_{\mathbb{R}}R_{i}\tilde{u}\,dx-2a_{i}\tilde{u}(\tilde{x}_{i}(t))\leq \|\tilde{u}\|_{L^\infty}\Big(\Big|\int_{\mathbb{R}}R_{i}\,dx\Big|+2a_i\Big)\leq O(\sqrt{\alpha}).
\end{aligned}
\end{eqnarray*}
Similarly,
\begin{align*}
<\partial^{2}_{x}S_{i},\tilde{v}>_{H^{-1},H^{1}}\leq O(\sqrt{\alpha}).
\end{align*}
For the term, $\int_{\mathbb{R}}\partial_{x}R_{i}\partial_{x}R_{j}\,dx$, we find
\begin{eqnarray*}
\begin{aligned}
\int_{\mathbb{R}}\partial_{x}R_{i}\partial_{x}R_{j}\,dx&=-\int_{\mathbb{R}}\partial_x^2 R_{i}R_{j}\,dx=-\int_{\mathbb{R}}\left(R_{i}-2a_{i}\delta(x-\tilde{x}_{i}(t))\right)R_{j}\,dx\\
&=-\int_{\mathbb{R}}R_{i}R_{j}\,dx+2a_{i}R_{j}(\tilde{x}_{i}(t)) \leq O(e^{-\frac L4}).
\end{aligned}
\end{eqnarray*}
Similarly,
\begin{align*}
\int_{\mathbb{R}}\partial_{x}S_{i}\partial_{x}S_{j}\,dx \leq O(e^{-\frac L4}).
\end{align*}

\begin{lemma}\label{implicit-1}
Assume $(\tilde{u},\tilde{v})$ satisfies \eqref{preestimate}, then we have
\begin{equation*}
\int_{\mathbb{R}}\Big[(\tilde{u}+\sum^{N}_{j=1}R_{j})(\tilde{v}+\sum^{N}_{j=1}S_{j})(\tilde{u}_{x}+\sum^{N}_{j=1}R_{jx})-\sum^{N}_{j=1}R_{j}S_{j}R_{jx}\Big]R_{ix}\,dx \leq O(\sqrt{\alpha})+O(e^{-\frac L4})
\end{equation*}
and
\begin{equation*}
\int_{\mathbb{R}}\Big[(\tilde{u}+\sum^{N}_{j=1}R_{j})(\tilde{v}+\sum^{N}_{j=1}S_{j})(\tilde{v}_{x}+\sum^{N}_{j=1}S_{jx})-\sum^{N}_{j=1}R_{j}S_{j}S_{jx}\Big]S_{ix}\,dx \leq O(\sqrt{\alpha})+O(e^{-\frac L4}).
\end{equation*}
\end{lemma}
\begin{proof}
We calculate
\begin{align*}
&\int_{\mathbb{R}}\Big[(\tilde{u}+\sum^{N}_{j=1}R_{j})(\tilde{v}+\sum^{N}_{j=1}S_{j})(\tilde{u}_{x}+\sum^{N}_{j=1}R_{jx})-\sum^{N}_{j=1}R_{j}S_{j}R_{jx}\Big]R_{ix}\,dx\\
=&\int_{\mathbb{R}}\Big[\tilde{u}(\tilde{v}+\sum^{N}_{j=1}S_{j})(\tilde{u}_{x}+\sum^{N}_{j=1}R_{jx})+\sum^{N}_{j=1}R_{j}\tilde{v}(\tilde{u}_{x}+\sum^{N}_{j=1}R_{jx})+\sum^{N}_{j=1}S_{j}\sum^{N}_{j=1}R_{j}\tilde{u}_{x}\\
&\qquad\qquad+ \sum^{N}_{j=1}S_{j}\sum^{N}_{j=1}R_{j}\sum^{N}_{j=1}R_{jx}-\sum^{N}_{j=1}R_{j}S_{j}R_{jx}\Big]R_{ix}\,dx\\
\leq &\Big|\int_{\mathbb{R}}\tilde{u}(\tilde{v}+\sum^{N}_{j=1}S_{j})(\tilde{u}_{x}+\sum^{N}_{j=1}R_{jx})R_{ix}\,dx\Big|+\Big|\int_{\mathbb{R}}\sum^{N}_{j=1}R_{j}\tilde{v}(\tilde{u}_{x}+\sum^{N}_{j=1}R_{jx})R_{ix}\,dx\Big|\\
&\quad+\Big|\int_{\mathbb{R}}\sum^{N}_{j=1}S_{j}\sum^{N}_{j=1}R_{j}\tilde{u}_{x}R_{ix}\,dx\Big|+\Big|\int_{\mathbb{R}}\Big[\sum^{N}_{j=1}S_{j}\sum^{N}_{j=1}R_{j}\sum^{N}_{j=1}R_{jx}-\sum^{N}_{j=1}R_{j}S_{j}R_{jx}\Big]R_{ix}\,dx\Big|\\
\leq &\|\tilde{u}\|_{L^{\infty}}\|u\|_{L^{\infty}}\|v\|_{L^{\infty}}\Big|\int_{\mathbb{R}}R_{ix}\,dx\Big|+\sum^{N}_{j=1}\|R_{j}\|_{L^{\infty}}\|\tilde{v}\|_{L^{\infty}}\|u\|_{L^{\infty}}\Big|\int_{\mathbb{R}}R_{ix}\,dx\Big|\\
&\quad+\Big(\sum^{N}_{j=1}\|R_{j}\|_{L^{\infty}}\Big)\Big(\sum^{N}_{j=1}\|S_{j}\|_{L^{\infty}}\Big)\Big(\int_{\mathbb{R}}(\tilde{u}_{x})^2\,dx\Big)^{\frac 12}\Big(\int_{\mathbb{R}}R_{ix}^2\,dx\Big)^{\frac 12}\\
&\quad\quad +\Big|\int_{\mathbb{R}}\Big[\sum_{i\ne j\,or\,i\ne k\,or\,j\ne k}R_{i}S_{j}R_{kx}\Big]R_{ix}\,dx\Big|\leq  O(\sqrt{\alpha})+O(e^{-\frac L4}).
\end{align*}
Similarly, we have
\begin{align*}
\int_{\mathbb{R}}\Big[(\tilde{u}+\sum^{N}_{j=1}R_{j})(\tilde{v}+\sum^{N}_{j=1}S_{j})(\tilde{v}_{x}+\sum^{N}_{j=1}S_{jx})-\sum^{N}_{j=1}R_{j}S_{j}S_{jx}\Big]S_{ix}\,dx \leq O(\sqrt{\alpha})+O(e^{-\frac L4}).
\end{align*}
\end{proof}

In order to estimate the next terms, we need the following lemma.

\begin{lemma}\label{implicit-2}
Under the same assumptions as in Lemma \eqref{implicit-1}, we have
\begin{equation*}
\begin{aligned}
\Big{\|} &P\ast \Big{(}\frac{1}{2}\Big{(}\tilde{u}_{x}+\sum^{N}_{j=1}R_{jx}\Big{)}^2(\tilde{v}+\sum^{N}_{j=1}S_{j})
+(\tilde{u}+\sum^{N}_{j=1}R_{j})\Big{(}\tilde{u}_{x}+\sum^{N}_{j=1}R_{jx}\Big{)}\Big{(}\tilde{v}_{x}+\sum^{N}_{j=1}S_{jx}\Big{)}\\
&+(\tilde{u}+\sum^{N}_{j=1}R_{j})^{2}(\tilde{v}+\sum^{N}_{j=1}S_{j})-\frac{1}{2}\sum^{N}_{j=1}R^{2}_{jx}S_{j}-\sum^{N}_{j=1}R_{j}R_{jx}S_{jx}-\sum^{N}_{j=1}R^{2}_{j}S_{j}\bigg{)}\Big{\|}_{L^\infty}\\
\leq & O(\sqrt{\alpha})+O(e^{-\frac L4})
\end{aligned}
\end{equation*}
and
\begin{equation*}
\begin{aligned}
\Big{\|} &P\ast \Big{(}\frac{1}{2}\Big{(}\tilde{v}_{x}+\sum^{N}_{j=1}S_{jx}\Big{)}^2(\tilde{u}+\sum^{N}_{j=1}R_{j})
+(\tilde{v}+\sum^{N}_{j=1}S_{j})\Big{(}\tilde{v}_{x}+\sum^{N}_{j=1}S_{jx}\Big{)}\Big{(}\tilde{u}_{x}+\sum^{N}_{j=1}R_{jx}\Big{)}\\
&+(\tilde{v}+\sum^{N}_{j=1}S_{j})^{2}(\tilde{u}+\sum^{N}_{j=1}R_{j})-\frac{1}{2}\sum^{N}_{j=1}S^{2}_{jx}R_{j}-\sum^{N}_{j=1}S_{j}S_{jx}R_{jx}-\sum^{N}_{j=1}S^{2}_{j}R_{j}\Big{)}\Big{\|}_{L^\infty}\\
\leq & O(\sqrt{\alpha})+O(e^{-\frac L4}).
\end{aligned}
\end{equation*}
\end{lemma}
\begin{proof}
By using H$\rm{\ddot{o}}$lder inequality and triangle inequality, we get
\begin{align*}
\Big{\|} &P\ast \Big{(}\frac{1}{2}\Big{(}\tilde{u}_{x}+\sum^{N}_{j=1}R_{jx}\Big{)}^2(\tilde{v}+\sum^{N}_{j=1}S_{j})
+(\tilde{u}+\sum^{N}_{j=1}R_{j})\Big{(}\tilde{u}_{x}+\sum^{N}_{j=1}R_{jx}\Big{)}\Big{(}\tilde{v}_{x}+\sum^{N}_{j=1}S_{jx}\Big{)}\\
&\;\;+(\tilde{u}+\sum^{N}_{j=1}R_{j})^{2}(\tilde{v}+\sum^{N}_{j=1}S_{j})-\frac{1}{2}\sum^{N}_{j=1}R^{2}_{jx}S_{j}-\sum^{N}_{j=1}R_{j}R_{jx}S_{jx}-\sum^{N}_{j=1}R^{2}_{j}S_{j}\Big{)}\Big{\|}_{L^\infty}\\
\leq &\frac{1}{2}\,\int_{\mathbb{R}}\Big|\Big{(}\tilde{u}_{x}+\sum^{N}_{j=1}R_{jx}\Big{)}^2(\tilde{v}+\sum^{N}_{j=1}S_{j})-\sum^{N}_{j=1}R^{2}_{jx}S_{j}\Big|\,dx\\
&\;\;+\int_{\mathbb{R}}\Big| (\tilde{u}+\sum^{N}_{j=1}R_{j})\Big{(}\tilde{u}_{x}+\sum^{N}_{j=1}R_{jx}\Big{)}\Big{(}\tilde{v}_{x}+\sum^{N}_{j=1}S_{jx}\Big{)}-\sum^{N}_{j=1}R_{j}R_{jx}S_{jx}\Big|\,dx\\
&\;\; +\int_{\mathbb{R}}\Big| (\tilde{u}+\sum^{N}_{j=1}R_{j})^{2}(\tilde{v}+\sum^{N}_{j=1}S_{j})-\sum^{N}_{j=1}R^{2}_{j}S_{j}\Big|\,dx\\
&=\frac{1}{2}\,I_{2,1}+I_{2,2}+I_{2,3}.
\end{align*}
For the term $I_{2,1}$, we have
\begin{align*}
I_{2,1}=&\int_{\mathbb{R}}\Big|\Big{(}\tilde{u}_{x}+\sum^{N}_{j=1}R_{jx}\Big{)}^2(\tilde{v}+\sum^{N}_{j=1}S_{j})-\sum^{N}_{j=1}R^{2}_{jx}S_{j}\Big|\,dx\\
\leq &\int_{\mathbb{R}}\Big|\Big[\Big{(}\tilde{u}_{x}+\sum^{N}_{j=1}R_{jx}\Big{)}^2-\sum^{N}_{j=1}R^{2}_{jx}\Big](\tilde{v}+\sum^{N}_{j=1}S_{j})\Big|\,dx+\int_{\mathbb{R}}\Big|\Big(\sum^{N}_{j=1}R^{2}_{jx}\Big)\tilde{v}\Big|\,dx\\
&\;\;+\int_{\mathbb{R}}\Big|\sum^{N}_{j=1}R^{2}_{jx}\sum^{N}_{i=1}S_{i}-\sum^{N}_{j=1}R^{2}_{jx}S_{j}\Big|\,dx\\
\leq &\|\tilde{v}\|_{L^\infty}\int_{\mathbb{R}}\Big( \Big|\tilde{u}^2_{x}\Big|+\Big|2\tilde{u}_{x}\sum^{N}_{j=1}R_{jx}\Big|+\Big|\Big(\sum^{N}_{j=1}R_{jx}\Big)^2-\sum^{N}_{j=1}R^2_{jx}\Big|\Big)\,dx\\
&\;\;+\|\tilde{v}\|_{L^\infty}\int_{\mathbb{R}}\sum^{N}_{j=1}R^{2}_{jx}\,dx+\int_{\mathbb{R}}\sum^{N}_{i\ne j}\sum^{N}_{j=1}R^{2}_{jx}S_{i}\,dx\leq O(\sqrt{\alpha})+O(e^{\frac L4}).
\end{align*}
For the term $I_{2,2}$, we have
\begin{align*}
I_{2,2}=&\int_{\mathbb{R}}\Big|(\tilde{u}+\sum^{N}_{j=1}R_{j})\Big{(}\tilde{u}_{x}+\sum^{N}_{j=1}R_{jx}\Big{)}\Big{(}\tilde{v}_{x}+\sum^{N}_{j=1}S_{jx}\Big{)}-\sum^{N}_{j=1}R_{j}R_{jx}S_{jx}\Big|\,dx\\
\leq &\int_{\mathbb{R}}\Big|\tilde{u}\Big{(}\tilde{u}_{x}+\sum^{N}_{j=1}R_{jx}\Big{)}\Big{(}\tilde{v}_{x}+\sum^{N}_{j=1}S_{jx}\Big{)}\Big|\,dx+\int_{\mathbb{R}}\Big|\Big(\sum^{N}_{j=1}R_{j}\Big)\tilde{u}_{x}\Big{(}\tilde{v}_{x}+\sum^{N}_{j=1}S_{jx}\Big{)}\Big|\,dx\\
&\;+\int_{\mathbb{R}}\Big|\Big(\sum^{N}_{j=1}R_{j}\Big)\Big(\sum^{N}_{j=1}S_{jx}\Big)\tilde{u}_{x}\Big|\,dx+\int_{\mathbb{R}}\Big|\Big(\sum^{N}_{j=1}R_{j}\Big)\Big(\sum^{N}_{j=1}S_{jx}\Big)\Big(\sum^{N}_{j=1}R_{jx}\Big)-\sum^{N}_{j=1}R_{j}R_{jx}S_{jx}\Big|\,dx\\
\leq
&\|\tilde{u}\|_{L^\infty}\int_{\mathbb{R}}\left|u_{x}v_{x}\right|\,dx+\|v\|_{L^\infty}\Big(\int_{\mathbb{R}}\Big(\sum^{N}_{j=1}R_{j}\Big)^2\,dx\Big)^{\frac{1}{2}}\Big(\int_{\mathbb{R}}\tilde{u}^2_{x}\,dx\Big)^{\frac 12}\\
&\;\;+\sum^{N}_{j=1}\|S_{j}\|_{L^\infty}\Big(\int_{\mathbb{R}}\Big(\sum^{N}_{j=1}R_{j}\Big)^2\,dx\Big)^{\frac 12}\Big(\int_{\mathbb{R}}\tilde{u}^2_{x}\,dx\Big)^{\frac 12}+\int_{\mathbb{R}}\Big|\sum_{i\ne j\,or\,j\ne k\,or\,i\ne k}R_{i}S_{jx}R_{jx}\Big|\,dx\\
\leq & O(\sqrt{\alpha})+O(e^{\frac L4}).
\end{align*}
For this term $I_{2,3}$, we obtain
\begin{align*}
I_{2,3}=&\int_{\mathbb{R}}\Big(\tilde{u}+\sum^{N}_{j=1}R_{j}\Big)^{2}\Big(\tilde{v}+\sum^{N}_{j=1}S_{j}\Big)-\sum^{N}_{j=1}R^{2}_{j}S_{j}\,dx\\
\leq & \int_{\mathbb{R}}\Big|\Big((\tilde{u}+\sum^{N}_{j=1}R_{j})^{2}-\sum^{N}_{j=1}R^{2}_{j}\Big)\Big(\tilde{v}+\sum^{N}_{j=1}S_{j}\Big)\Big|\,dx+\int_{\mathbb{R}}\Big|\sum^{N}_{j=1}R^2_{j}\tilde{v}\Big|\,dx\\
&\qquad\qquad +\int_{\mathbb{R}}\Big|\sum^{N}_{j=1}R^{2}_{j}\sum^{N}_{i=1}S_{i}-\sum^{N}_{j=1}R^{2}_{j}S_{j}\Big|\,dx\\
\leq &\|v\|_{L^\infty}\int_{\mathbb{R}}|\tilde{u}|^2+\Big|2\tilde{u}\sum^{N}_{j=1}R_{j}\Big|+\Big|\Big(\sum^{N}_{j=1}R_{j}\Big)^2-\sum^{N}_{j=1}R^{2}_{j}\Big|\,dx+\|\tilde{v}\|_{L^\infty}\int_{\mathbb{R}}\sum^{N}_{j=1}R^{2}_{j}\,dx\\
&\qquad\qquad+\int_{\mathbb{R}}\sum^{N}_{i\ne j}\sum^{N}_{j=1}R^{2}_{j}S_{i}\,dx\leq O(\sqrt{\alpha})+O(e^{\frac L4}).
\end{align*}
Accordingly, we get
\begin{align*}
\Big{\|} &P\ast \Big{(}\frac{1}{2}\Big{(}\tilde{u}_{x}+\sum^{N}_{j=1}R_{jx}\Big{)}^2(\tilde{v}+\sum^{N}_{j=1}S_{j})
+(\tilde{u}+\sum^{N}_{j=1}R_{j})\Big{(}\tilde{u}_{x}+\sum^{N}_{j=1}R_{jx}\Big{)}\Big{(}\tilde{v}_{x}+\sum^{N}_{j=1}S_{jx}\Big{)}\\
&\qquad\qquad+(\tilde{u}+\sum^{N}_{j=1}R_{j})^{2}(\tilde{v}+\sum^{N}_{j=1}S_{j})-\frac{1}{2}\sum^{N}_{j=1}R^{2}_{jx}S_{j}-\sum^{N}_{j=1}R_{j}R_{jx}S_{jx}-\sum^{N}_{j=1}R^{2}_{j}S_{j}\Big{)}\Big{\|}_{L^\infty}\\
\leq & O(\sqrt{\alpha})+O(e^{-\frac L4}).
\end{align*}
Similarly, we can prove
\begin{align*}
\Big{\|} &P\ast \Big{(}\frac{1}{2}\Big{(}\tilde{v}_{x}+\sum^{N}_{j=1}S_{jx}\Big{)}^2(\tilde{u}+\sum^{N}_{j=1}R_{j})
+(\tilde{v}+\sum^{N}_{j=1}S_{j})\Big{(}\tilde{v}_{x}+\sum^{N}_{j=1}S_{jx}\Big{)}\Big{(}\tilde{u}_{x}+\sum^{N}_{j=1}R_{jx}\Big{)}\\
&\qquad\qquad\qquad+(\tilde{v}+\sum^{N}_{j=1}S_{j})^{2}(\tilde{u}+\sum^{N}_{j=1}R_{j})-\frac{1}{2}\sum^{N}_{j=1}S^{2}_{jx}R_{j}-\sum^{N}_{j=1}S_{j}S_{jx}R_{jx}-\sum^{N}_{j=1}S^{2}_{j}R_{j}\Big{)}\Big{\|}_{L^\infty}\\
\leq & O(\sqrt{\alpha})+O(e^{-\frac L4}).
\end{align*}
This completes the proof of this lemma.
\end{proof}

Thanks to the lemma \ref{implicit-2}, we have
\begin{align*}
\int_{\mathbb{R}}&P\ast\Big{(}\frac{1}{2}\Big{(}\tilde{u}_{x}+\sum^{N}_{j=1}R_{jx}\Big{)}^2(\tilde{v}+\sum^{N}_{j=1}S_{j})
+(\tilde{u}+\sum^{N}_{j=1}R_{j})\Big{(}\tilde{u}_{x}+\sum^{N}_{j=1}R_{jx}\Big{)}\Big{(}\tilde{v}_{x}+\sum^{N}_{j=1}S_{jx}\Big{)}\\
&+(\tilde{u}+\sum^{N}_{j=1}R_{j})^{2}(\tilde{v}+\sum^{N}_{j=1}S_{j})-\frac{1}{2}\sum^{N}_{j=1}R^{2}_{jx}S_{j}-\sum^{N}_{j=1}R_{j}R_{jx}S_{jx}-\sum^{N}_{j=1}R^{2}_{j}S_{j}\Big{)}\partial^{2}_{x}R_{i}\,dx\\
& \triangleq \int_{\mathbb{R}}A(x)\partial^{2}_{x}R_{i}\,dx=\int_{\mathbb{R}}A(x)\left(R_{i}-2a_{i}\delta(x-\tilde{x}_{i}(t))\right)\,dx\\
=&\int_{\mathbb{R}}A(x)R_{i}\,dx-2a_{i}A(\tilde{x}_{i}(t))\leq O\left(\|A\|_{L^\infty}\right) \leq O(\sqrt{\alpha})+O(e^{-\frac L4}).
\end{align*}
Similarly, we get
\begin{align*}
\int_{\mathbb{R}}&P\ast\Big{(}\frac{1}{2}\Big{(}\tilde{v}_{x}+\sum^{N}_{j=1}S_{jx}\Big{)}^2(\tilde{u}+\sum^{N}_{j=1}R_{j})
+(\tilde{v}+\sum^{N}_{j=1}S_{j})\Big{(}\tilde{v}_{x}+\sum^{N}_{j=1}S_{jx}\Big{)}\Big{(}\tilde{u}_{x}+\sum^{N}_{j=1}R_{jx}\Big{)}\\
&+(\tilde{v}+\sum^{N}_{j=1}S_{j})^{2}(\tilde{u}+\sum^{N}_{j=1}R_{j})-\frac{1}{2}\sum^{N}_{j=1}S^{2}_{jx}R_{j}-\sum^{N}_{j=1}S_{j}S_{jx}R_{jx}-\sum^{N}_{j=1}S^{2}_{j}R_{j}\Big{)}\partial^{2}_{x}S_{i}\,dx\\
\leq & O(\sqrt{\alpha})+O(e^{-\frac L4}).
\end{align*}

\begin{lemma}\label{implicit-3}
Under the same assumption as \eqref{implicit-1}, we have
\begin{align*}
&\Big\|P\ast\Big{(}\frac{1}{2}\big{(}\tilde{u}_{x}+\sum^{N}_{j=1}R_{jx}\big{)}^{2}\big{(}\tilde{v}_{x}+\sum^{N}_{j=1}S_{jx}\big{)}-\frac{1}{2}\sum^{N}_{j=1}R^{2}_{jx}S_{jx}\Big{)}\Big\|_{L^\infty}\leq O(\sqrt{\alpha})+O(e^{-\frac L4}),\\
&\Big\|P\ast\Big{(}\frac{1}{2}\big{(}\tilde{v}_{x}+\sum^{N}_{j=1}S_{jx}\big{)}^{2}\big{(}\tilde{u}_{x}+\sum^{N}_{j=1}R_{jx}\big{)}-\frac{1}{2}\sum^{N}_{j=1}S^{2}_{jx}R_{jx}\Big{)}\Big\|_{L^\infty}\leq O(\sqrt{\alpha})+O(e^{-\frac L4}).
\end{align*}
\end{lemma}
\begin{proof}
We estimate
\begin{align*}
&\Big\|P\ast\Big{(}\frac{1}{2}\big{(}\tilde{u}_{x}+\sum^{N}_{j=1}R_{jx}\big{)}^{2}\big{(}\tilde{v}_{x}+\sum^{N}_{j=1}S_{jx}\big{)}-\frac{1}{2}\sum^{N}_{j=1}R^{2}_{jx}S_{jx}\Big{)}\Big\|\\
\leq & \frac{1}{2}\,\int_{\mathbb{R}}\Big|(\tilde{u}_{x}+\sum^{N}_{j=1}R_{jx})^{2}(\tilde{v}_{x}+\sum^{N}_{j=1}S_{jx})-\sum^{N}_{j=1}R^{2}_{jx}S_{jx}\Big|\,dx\\
\leq &
\frac{1}{2}\,\Big(\int_{\mathbb R}\Big|\Big((\tilde{u}_{x}+\sum^{N}_{j=1}R_{jx})^{2}-\sum^{N}_{j=1}R^{2}_{jx}\Big)(\tilde{v}_{x}+\sum^{N}_{j=1}S_{jx})\Big|\,dx+\int_{\mathbb R}\Big|\tilde{v}_{x}\sum^{N}_{j=1}R^{2}_{jx}\Big|\,dx\Big)\\
&\qquad\qquad +\frac{1}{2}\int_{\mathbb R}\Big|\sum^{N}_{j=1}R^{2}_{jx}\sum^{N}_{j=1}S_{jx}-\sum^{N}_{j=1}R^{2}_{jx}S_{jx}\Big|\,dx=\frac{1}{2}\,(I_{3,1}+I_{3,2}+I_{3,3}).
\end{align*}
For the term $I_{3,1}$, we obtain
\begin{align*}
I_{3,1}=&\int_{\mathbb{R}}\Big|\Big{(}\tilde{u}_{x}+\sum^{N}_{j=1}R_{jx}\Big{)}^2(\tilde{v}_{x}+\sum^{N}_{j=1}S_{jx})-\sum^{N}_{j=1}R^{2}_{jx}S_{jx}\Big|\,dx\\
\leq &\int_{\mathbb{R}}\Big|\Big[\Big{(}\tilde{u}_{x}+\sum^{N}_{j=1}R_{jx}\Big{)}^2-\sum^{N}_{j=1}R^{2}_{jx}\Big](\tilde{v}_{x}+\sum^{N}_{j=1}S_{jx})\Big|\,dx+\int_{\mathbb{R}}\Big|\Big(\sum^{N}_{j=1}R^{2}_{jx}\Big)\tilde{v}_{x}\Big|\,dx\\
&\qquad\qquad+\int_{\mathbb{R}}\Big|\sum^{N}_{j=1}R^{2}_{jx}\sum^{N}_{i=1}S_{ix}-\sum^{N}_{j=1}R^{2}_{jx}S_{jx}\Big|\,dx\\
\leq &\|v\|_{L^\infty}\int_{\mathbb{R}}\Big|\tilde{u}^2_{x}\Big|+\Big|2\tilde{u}_{x}\sum^{N}_{j=1}R_{jx}\Big|+\Big|\Big(\sum^{N}_{j=1}R_{jx}\Big)^2-\sum^{N}_{j=1}R^2_{jx}\Big|\,dx\\
&\qquad\qquad+\|\tilde{v}\|_{L^\infty}\int_{\mathbb{R}}\sum^{N}_{j=1}R^{2}_{jx}\,dx+\int_{\mathbb{R}}\sum^{N}_{i\ne j}\sum^{N}_{j=1}\left|R^{2}_{jx}S_{ix}\right|\,dx\leq O(\sqrt{\alpha})+O(e^{-\frac L4}).
\end{align*}
For the term $I_{3,2}$, we get
\begin{align*}
I_{3,2}=\int_{\mathbb R}\sum^{N}_{j=1}R^2_{jx}|\tilde{v}_{x}|\,dx= \sum^{N}_{j=1}\int_{\mathbb R}R^2_{jx}|\tilde{v}_{x}|\,dx\leq \sum^{N}_{j=1}\Big(\int_{\mathbb R}\tilde{v}^2_{x}\,dx\Big)^{\frac 12}\Big(\int_{\mathbb R}R^4_{jx}\,dx\Big)^{\frac 12}\leq O(\sqrt{\alpha}).
\end{align*}
For the term $I_{3,3}$, we have
\begin{align*}
 \int_{\mathbb R}\Big|\sum^{N}_{j=1}R^{2}_{jx}\sum^{N}_{j=1}S_{jx}-\sum^{N}_{j=1}R^{2}_{jx}S_{jx}\Big|\,dx \leq \int_{\mathbb R}\sum^{N}_{i\ne j}\sum^{N}_{j=1}\left|R^{2}_{jx}S_{ix}\right|\,dx \leq  O(e^{-\frac L4}).
\end{align*}
Thus, we deduce that
\begin{align*}
\Big\|P\ast\Big{(}\frac{1}{2}\big{(}\tilde{u}_{x}+\sum^{N}_{j=1}R_{jx}\big{)}^{2}\big{(}\tilde{v}_{x}+\sum^{N}_{j=1}S_{jx}\big{)}-\frac{1}{2}\sum^{N}_{j=1}R^{2}_{jx}S_{jx}\Big{)}\Big\|_{L^\infty}\leq O(\sqrt{\alpha})+O(e^{-\frac L4})
\end{align*}
and
\begin{align*}
&\Big\|P\ast\Big{(}\frac{1}{2}\big{(}\tilde{v}_{x}+\sum^{N}_{j=1}S_{jx}\big{)}^{2}\big{(}\tilde{u}_{x}+\sum^{N}_{j=1}R_{jx}\big{)}-\frac{1}{2}\sum^{N}_{j=1}S^{2}_{jx}R_{jx}\Big{)}\Big\|_{L^\infty}\leq O(\sqrt{\alpha})+O(e^{-\frac L4}).
\end{align*}
This completes the proof of this lemma.
\end{proof}

On account of Lemma \eqref{implicit-3}, we obtain
\begin{align*}
&\int_{\mathbb{R}}P\ast\Big{(}\frac{1}{2}\big{(}\tilde{u}_{x}+\sum^{N}_{j=1}R_{jx}\big{)}^{2}\big{(}\tilde{v}_{x}+\sum^{N}_{j=1}S_{jx}\big{)}-\frac{1}{2}\sum^{N}_{j=1}R^{2}_{jx}S_{jx}\Big{)}\partial_{x}R_{i}\,dx\\
\leq &\Big( O(\sqrt{\alpha})+O(e^{\frac L4})\Big) \int_{\mathbb R}|R_{ix}|\,dx \leq O(\sqrt{\alpha})+O(e^{-\frac L4}).
\end{align*}
Similarly, we find
\begin{align*}
\int_{\mathbb{R}}P\ast\Big{(}\frac{1}{2}\big{(}\tilde{v}_{x}+\sum^{N}_{j=1}S_{jx}\big{)}^{2}\big{(}\tilde{u}_{x}+\sum^{N}_{j=1}R_{jx}\big{)}-\frac{1}{2}\sum^{N}_{j=1}S^{2}_{jx}R_{jx}\Big{)}\partial_{x}S_{i}\,dx\leq O(\sqrt{\alpha})+O(e^{-\frac L4}).
\end{align*}
Thanks to the lemmas \ref{implicit-1}, \ref{implicit-2} and \ref{implicit-3}, we arrive at
\begin{align*}
|\dot{\tilde{x}}_{i}(t)-c_{i}|\Big{(}\|R_{ix}\|^{2}_{L^{2}}+\|S_{ix}\|^{2}_{L^{2}}+O(\sqrt{\alpha})\Big{)}\leq O(\sqrt{\alpha})+O(e^{-\frac{L}{8}}).
\end{align*}
Since $\|R_{ix}\|_{L^2}>a_1$ and $\|S_{ix}\|^{2}_{L^2}>b_1$, then we have
\begin{align*}
|\dot{\tilde{x}}_{i}(t)-c_{i}|\leq O(\sqrt{\alpha})+O(e^{-\frac{L}{8}}).
\end{align*}

Finally, we claim
\begin{align*}
\left|x_{i}-\tilde{x}_{i}\right|\leq \frac{L}{12}.
\end{align*}
Indeed, if $x\notin [\tilde{x}_{i}-L/12,\tilde{x}_{i}+L/12]$, then
\begin{align*}
u(x,t)v(x,t)\leq c_{i}e^{-\frac{L}{6}}+O(\sqrt{\alpha})+O(e^{-\frac{L}{4}}) \leq c_{i}-O(\sqrt{\alpha})-O(e^{-\frac{L}{4}}),
\end{align*}
by choosing $\alpha$ small enough and $L$ large enough.
However,
\begin{align*}
u(t,x_{i})v(t,x_{i})=\max_{x\in J_{i}(t)}u(x)v(x)\geq u(\tilde{x}_{i})v(\tilde{x}_{i})=c_{i}+O(\sqrt{\alpha})+O(e^{-\frac{L}{4}}),
\end{align*}
which is a contradiction.
Therefore, we have $\left|x_{i}-\tilde{x}_{i}\right|\leq L/12$.
\end{proof}

\subsection{Monotonicity property}
Thanks to the preceding proposition, for $\epsilon_{0}>0$ small enough and $L_{0}>0$ large enough, one can construct $C^1$ functions $\tilde{x}_{1},...,\tilde{x}_{N}$ defined on $[0,t_{0}]$ such that \eqref{initial-3-1}-\eqref{initial-3-5} are satisfied. In this subsection we investigate the almost monotonicity of functionals that are very close to the energy at the right of $i$th bump, $i=1,...,N-1$ of $(u, v)$. Let $\Psi$ be a $C^{\infty}$-function such that
\begin{equation*}
\left\{
\begin{aligned}
&0<\Psi(x)<1,\,\Psi'(x)>0,\quad\quad\quad & x\in{\mathbb R},\\
&|\Psi'''|\leq 10|\Psi'|,&x\in [-1,1],\\
\end{aligned}
\right.
\end{equation*}
and
\begin{equation*}
\Psi(x)=\left\{
\begin{aligned}
&e^{-|x|},\quad &x<-1,\\
&1-e^{-|x|},&x>1.
\end{aligned}
\right.
\end{equation*}
Setting $\Psi_{K}=\Psi(\cdot/K)$, we introduce for $i=2,...,N$,
\begin{align*}
&\mathcal{J}^{u}_{j,K}(t)=\int_{\mathbb R}\left(u^2(t)+u^2_{x}(t)\right)\Psi_{j,K}(t)\,dx,\quad \mathcal{J}^{v}_{j,K}(t)=\int_{\mathbb R}\left(v^2(t)+v^2_{x}(t)\right)\Psi_{j,K}(t)\,dx,\\
&\mathcal{J}^{u,v}_{j,K}(t)=\int_{\mathbb R}\left(u(t)v(t)+u_{x}(t)v_{x}(t)\right)\Psi_{j,K}(t)\,dx,
\end{align*}
where $\Psi_{j,K}(t,x)=\Psi_K(x-y_{j}(t))$ with $y_{j}(t)$, $j=2,...,N$, defined in \eqref{initial-3-4}. Note that $\mathcal{J}^{u}_{j,K}(t)$ is close to $\left\|u(t)\right\|_{H^{1}(x>y_{j}(t))}$ and thus measures the energy at the right of the $(j-1)th$ bump of $u$, $\mathcal{J}^{v}_{j,K}(t)$ is close to $\left\|v(t)\right\|_{H^{1}(x>y_{j}(t))}$ and thus measures the energy at the right of the $(j-1)$th bump of $v$, and $\mathcal{J}^{u,v}_{j,K}(t)$ is close to $<u(t),v(t)>_{H^{1}\times H^{1}(x>y_{j}(t))}$ and thus measures the energy at the right of the $(j-1)th$ bump of $(u, v)$. Finally, we set
\begin{align*}
\sigma_{0}=\frac{1}{4}\,\min(c_{1},c_{2}-c_{1},...,c_{N}-c_{N-1}).
\end{align*}

We have the following monotonicity result.
\begin{prop}\label{monotonicity}
(Exponential decay of the functionals $\mathcal{J}^{u}_{j,K}(t)$, $\mathcal{J}^{v}_{j,K}(t)$ and $\mathcal{J}^{u,v}_{j,K}(t)$). Let $(u,v)\in Y([0,T[)$ be a solution of two component Novikov equations satisfying \eqref{initial-3-2} on $[0,t_{0}]$. There exist $\alpha_{0}>0$ and $L_{0}>0$ only depending on $c_{1}$ such that if $0<\alpha<\alpha_{0}$ and $L\geq L_{0}$ then for any $4 \leq K \lesssim \sqrt{L}$,
\begin{align*}
&\mathcal{J}^{u}_{j,K}(t)-\mathcal{J}^{u}_{j,K}(0)\leq O(e^{-\frac{\sigma_{0}L}{8K}}), \quad \mathcal{J}^{v}_{j,K}(t)-\mathcal{J}^{v}_{j,K}(0)\leq O(e^{-\frac{\sigma_{0}L}{8K}}),\\
&\mathcal{J}^{u,v}_{j,K}(t)-\mathcal{J}^{u,v}_{j,K}(0)\leq O(e^{-\frac{\sigma_{0}L}{8K}}),\quad \forall j\in \{2,...,N\},\;\forall t\in[0,t_{0}].
\end{align*}
\end{prop}

The proof of this Proposition relies on the following Virial type identity.

\begin{lemma}\label{virial}
(Virial type identity). Let $(u,v)\in Y([0,T[)$, with $0<T\leq +\infty$, be a solution of equation \eqref{tcNK} that satisfies \eqref{initialdata-1}. For any smooth space function $g:\mathbb{R}\rightarrow\mathbb{R}$, there holds
\begin{align*}
\frac{d}{dt}\,&\int_{\mathbb R}\left(u^{2}+u^{2}_{x}\right)g\,dx\\
=&\int_{\mathbb R}u\Big[u^2_{x}v+2P\ast\Big(\frac{1}{2}u^{2}_{x}v+uu_{x}v_{x}+u^{2}v\Big)+2P_{x}\ast\Big(\frac{1}{2}u^{2}_{x}v_{x}\Big)\Big]g'\,dx+\int_{\mathbb R}\left(u^{2}+u^{2}_{x}\right)g'\,dx,\\
\frac{d}{dt}\,&\int_{\mathbb R}\left(v^{2}+v^{2}_{x}\right)g\,dx\\
=&\int_{\mathbb R}v\Big[v^2_{x}u+2P\ast\Big(\frac{1}{2}v^{2}_{x}u+vv_{x}u_{x}+v^{2}u\Big)+2P_{x}\ast\Big(\frac{1}{2}v^{2}_{x}u_{x}\Big)\Big]g'\,dx+\int_{\mathbb R}\left(v^{2}+v^{2}_{x}\right)g'\,dx,\\
{\rm and}\quad &\frac{d}{dt}\,\int_{\mathbb R}\left(uv+u_{x}v_{x}\right)g\,dx\\
=&\int_{\mathbb R}u\Big[vu_{x}v_{x}+P\ast\Big(\frac{1}{2}v^{2}_{x}u+vv_{x}u_{x}+v^{2}u\Big)+P_{x}\ast\left(\frac{1}{2}v^{2}_{x}u_{x}\right)\Big]g'\,dx\\
&\;\;+\int_{\mathbb R}v\Big[P\ast\Big(\frac{1}{2}u^{2}_{x}v+uu_{x}v_{x}+u^{2}v\Big)+P_{x}\ast\Big(\frac{1}{2}u^{2}_{x}v_{x}\Big)\Big]g'\,dx+\int_{\mathbb R}\left(uv+u_{x}v_{x}\right)g'\,dx.
\end{align*}
\end{lemma}
\begin{proof}
By using the weak form of $u$, integrating by parts, we calculate
\begin{align*}
\frac{d}{dt}\,&\int_{\mathbb R}\left(u^{2}+u^{2}_{x}\right)g\,dx=2\int_{\mathbb R}\left(uu_{t}+u_{x}u_{xt}\right)g\,dx+\int_{\mathbb R}\left(u^{2}+u^{2}_{x}\right)g'\,dx\\
=&2\int_{\mathbb R}\left(uu_{t}-uu_{xxt}-uu_{xt}\right)g'\,dx+\int_{\mathbb R}\left(u^{2}+u^{2}_{x}\right)g'\,dx\\
=&\int_{\mathbb R}\Big(2u^2vm-2u\Big[u^2v-\frac{1}{2}u^2_{x}v-uvu_{xx}-P\ast\Big(\frac{1}{2}\,u^2_{x}v+uu_{x}v_{x}+u^2v\Big)-P_{x}\ast\Big(\frac{1}{2}\,u^2_{x}v_{x}\Big)\Big]\Big)g'\,dx\\
&\;\;+\int_{\mathbb R}\left(u^{2}+u^{2}_{x}\right)g'\,dx\\
=&\int_{\mathbb R}\Big[uu^2_{x}v+2uP\ast\Big(\frac{1}{2}u^{2}_{x}v+uu_{x}v_{x}+u^{2}v\Big)+2uP_{x}\ast\Big(\frac{1}{2}u^{2}_{x}v_{x}\Big)\Big]g'\,dx+\int_{\mathbb R}\left(u^{2}+u^{2}_{x}\right)g'\,dx.
\end{align*}
Similarly, we have
\begin{eqnarray*}
\begin{aligned}
\frac{d}{dt}\,&\int_{\mathbb R}\left(v^{2}+v^{2}_{x}\right)g\,dx=\int_{\mathbb R}\Big[vv^2_{x}u+2vP\ast\Big(\frac{1}{2}v^{2}_{x}u+vv_{x}u_{x}+v^{2}u\Big)+2vP_{x}\ast\Big(\frac{1}{2}v^{2}_{x}u_{x}\Big)\Big]g'\,dx\\
&\qquad\quad\qquad\quad\qquad\quad+\int_{\mathbb R}\left(v^{2}+v^{2}_{x}\right)g'\,dx.\\
\end{aligned}
\end{eqnarray*}
Using the weak form of $(u,v)$ and its first derivative, integrating by parts, we obtain
\begin{align*}
\frac{d}{dt}\,&\int_{\mathbb R}\left(uv+u_{x}v_{x}\right)g\,dx\\
=&\int_{\mathbb R}\left(uv_{t}+u_{t}v-u_{xxt}v-uv_{xxt}-u_{xt}v-uv_{xt}\right)g\,dx+\int_{\mathbb R}\left(uv+u_{x}v_{x}\right)g'\,dx\\
=&\int_{\mathbb R}\left(un+mv-u_{xt}v-uv_{xt}\right)g\,dx+\int_{\mathbb R}\left(uv+u_{x}v_{x}\right)g'\,dx\\
=&\int_{\mathbb R}\Big[uvu_{x}v_{x}+uP\ast\Big(\frac{1}{2}v^{2}_{x}u+vv_{x}u_{x}+v^{2}u\Big)+uP_{x}\ast\Big(\frac{1}{2}v^{2}_{x}u_{x}\Big)\Big]g'\,dx\\
&\;\;+\int_{\mathbb R}\Big[vP\ast\Big(\frac{1}{2}u^{2}_{x}v+uu_{x}v_{x}+u^{2}v\Big)+vP_{x}\ast\Big(\frac{1}{2}u^{2}_{x}v_{x}\Big)\Big]g'\,dx+\int_{\mathbb R}\left(uv+u_{x}v_{x}\right)g'\,dx.
\end{align*}
\end{proof}

\begin{proof} [Proof of Proposition 4.2.]
We first note that, combining \eqref{initial-3-1} and \eqref{initial-3-4}, it holds for $i=2,..,N$,
\begin{eqnarray}\label{infspeed}
\begin{aligned}
\dot{y}_{i}(t)=\frac{\dot{\tilde{x}}_{i-1}(t)+\dot{\tilde{x}}_{i}(t)}{2}=\frac{c_{i-1}+c_{i}}{2}+O(\sqrt{\alpha})+O(L^{-1})\geq \frac{c_1}{2}.
\end{aligned}
\end{eqnarray}
Recall that the assumption ensures that $u\geq 0$ and $v\geq 0$ on $\mathbb R$. Now, applying the Virial type identity with $g=\Psi_{i,K}$ and using \eqref{infspeed}, we get
\begin{align*}
\frac{d}{dt}\,\mathcal{J}^{u}_{i,K}(t)=&-\dot{y}_{i}\,\int_{\mathbb R}\left(u^{2}+u^{2}_{x}\right)\Psi'_{i,K}\,dx\\
&\;\;+\int_{\mathbb R}\Big[uu^2_{x}v+2uP\ast\Big(\frac{1}{2}u^{2}_{x}v+uu_{x}v_{x}+u^{2}v\Big)+2uP_{x}\ast\Big(\frac{1}{2}u^{2}_{x}v_{x}\Big)\Big]\Psi'_{i,K}\,dx\\
\leq &-\frac{c_1}{2}\int_{\mathbb R}\left(u^{2}+u^{2}_{x}\right)\Psi'_{i,K}\,dx+\int_{\mathbb R}uu^2_{x}v\Psi'_{i,K}\,dx\\
&\;\;+\int_{\mathbb R}\Big[2uP\ast\Big(\frac{1}{2}u^{2}_{x}v+uu_{x}v_{x}+u^{2}v\Big)\Big]\Psi'_{i,K}\,dx+\int_{\mathbb R}\Big[2uP_{x}\ast\Big(\frac{1}{2}u^{2}_{x}v_{x}\Big)\Big]\Psi'_{i,K}\,dx\\
\triangleq &-\frac{c_1}{2}\int_{\mathbb R}\left(u^{2}+u^{2}_{x}\right)\Psi'_{i,K}\,dx+\widetilde{I}_{1,1}+\widetilde{I}_{1,2}+\widetilde{I}_{1,3}.
\end{align*}
We claim that for the terms $\tilde{I}_{1,i}$, $i=1,2,3$, it holds
\begin{align*}
\widetilde{I}_{1,i}\leq \frac{c_1}{20}\int_{\mathbb R}\left(u^{2}+u^{2}_{x}\right)\Psi'_{i,K}\,dx+\frac{C}{K}\|u_{0}\|^{3}_{H^{1}(\mathbb R)}\|v_{0}\|_{H^{1}(\mathbb R)}e^{-\frac{1}{K}(\sigma_{0}t+\frac{L}{8})}.
\end{align*}
Indeed, we divide $\mathbb R$ into two regions $D_{i}$ and $D^{c}_{i}$ with
\begin{align*}
D_{i}=\left[\tilde{x}_{i-1}(t)+\frac{L}{4},\tilde{x}_{i}(t)-\frac{L}{4}\right],\;i=2,...,N.
\end{align*}
Combining \eqref{initial-3-3} and \eqref{initial-3-4}, one can check that for $x\in D^c_{i}$,
\begin{eqnarray}\label{Estimate}
\begin{aligned}
\left|x-y_{i}(t)\right|\geq \frac{\tilde{x}_{i}(t)-\tilde{x}_{i-1}(t)}{2}-\frac{L}{4}\geq \frac{c_{i}-c_{i-1}}{4}\,t+\frac{L}{8}\geq \sigma_{0}t+\frac{L}{8}.
\end{aligned}
\end{eqnarray}
Let us begin by an estimate of $\widetilde{I}_{1,1}$. Using \eqref{Estimate}, Sobolev imbedding and the exponential decay of $\Psi'_{i,K}$ on $D^c_{i}$, we get
\begin{align*}
\widetilde{I}_{1,1}=&\int_{D_{i}}uu^2_{x}v\Psi'_{i,K}\,dx+\int_{D^c_{i}}uu^2_{x}v\Psi'_{i,K}\,dx\\
\leq &\|u\|_{L^\infty(D_{i})}\|v\|_{L^\infty(D_{i})}\int_{D_{i}}\left(u^2_{x}+u^2\right)\Psi'_{i,K}\,dx+\|\Psi'_{i,K}\|_{L^\infty(D^{c}_{i})}\|u\|_{L^\infty(\mathbb R)}\|v\|_{L^\infty(\mathbb R)}\|u\|^2_{H^1}.
\end{align*}
Now, using the exponential decay of $\varphi_{c}$ and $\psi_{c}$ on $D_{i}$, it holds
\begin{eqnarray}\label{predict-1}
\begin{aligned}
&\|u\|_{L^\infty(D_{i})}\leq \Big\|u-\sum^{N}_{j=1}\varphi(\cdot-\tilde{x}_{j}(t))\Big\|_{L^\infty(D_{i})}+\sum^{N}_{j=1}\left\|\varphi(\cdot-\tilde{x}_{j}(t))\right\|_{L^\infty(D_{i})}\leq  O(\sqrt{\alpha})+O(e^{-\frac{L}{8}}),\\
& \|v\|_{L^\infty(D_{i})}\leq \Big\|v-\sum^{N}_{j=1}\psi(\cdot-\tilde{x}_{j}(t))\Big\|_{L^\infty(D_{i})}+\sum^{N}_{j=1}\left\|\psi(\cdot-\tilde{x}_{j}(t))\right\|_{L^\infty(D_{i})}\leq  O(\sqrt{\alpha})+O(e^{-\frac{L}{8}}).
\end{aligned}
\end{eqnarray}
Therefore, for $0<\alpha<\alpha_{0}$ and $L>L_{0}>0$, with $\alpha_{0}\ll 1$ and $L_{0}\gg 1$, we obtain
\begin{align*}
\widetilde{I}_{1,1} \leq \frac{c_1}{20}\int_{\mathbb R}\left(u^{2}+u^{2}_{x}\right)\Psi'_{i,K}\,dx+\frac{C}{K}\|u_{0}\|^{3}_{H^{1}(\mathbb R)}\|v_{0}\|_{H^{1}(\mathbb R)}e^{-\frac{1}{K}(\sigma_{0}t+\frac{L}{8})}.
\end{align*}
Before estimating the term $\widetilde{I}_{1,2}$, using that $|\Psi'''_{i,K}|\leq 10K^{-2}\Psi'_{i,K}$, we have
\begin{align*}
\left(1-\partial^2_{x}\right)\Psi'_{i,K}(x)=\Psi'_{i,K}(x)-\frac{1}{K^2}\Psi'''_{i,K}(x)\geq \left(1-\frac{10}{K^2}\right)\Psi'_{i,K}(x),\; \forall x\in \mathbb R,
\end{align*}
and since $K>4$, it holds
\begin{eqnarray}\label{convolutionESI}
\begin{aligned}
\left(1-\partial^2_{x}\right)^{-1}\Psi'_{i,K}(x)\leq \Big(1-\frac{10}{K^2}\Big)^{-1}\Psi'_{i,K}(x),\; \forall x\in \mathbb R.
\end{aligned}
\end{eqnarray}
Next, the estimate of $\widetilde{I}_{1,2}$ gives us
\begin{align*}
\widetilde{I}_{1,2}=&\int_{D_{i}}\Big[2uP\ast\Big(\frac{1}{2}u^{2}_{x}v+uu_{x}v_{x}+u^{2}v\Big)\Big]\Psi'_{i,K}\,dx+\int_{D^c_{i}}\Big[2uP\ast\Big(\frac{1}{2}u^{2}_{x}v+uu_{x}v_{x}+u^{2}v\Big)\Big]\Psi'_{i,K}\,dx\\
\leq\; & \|u\|_{L^\infty(D_{i})}\int_{\mathbb R}P\ast\left|u^{2}_{x}v+2uu_{x}v_{x}+2u^{2}v\right|\Psi'_{i,K}\,dx\\
&\quad +\|\Psi'_{i,K}\|_{L^\infty(D^c_{i})}\|u\|_{L^\infty(\mathbb R)}\int_{\mathbb R}\left|u^{2}_{x}v+2uu_{x}v_{x}+2u^{2}v\right|\,dx\\
\leq\; & 3\|u\|_{L^\infty(D_{i})}\|v\|_{L^\infty(D_{i})}\int_{\mathbb R}\left|u^{2}_{x}+u^{2}\right|\Psi'_{i,K}\,dx+\frac{C}{K}\|u_{0}\|^{3}_{H^{1}(\mathbb R)}\|v_{0}\|_{H^{1}(\mathbb R)}e^{-\frac{1}{K}(\sigma_{0}t+\frac{L}{8})}\\
\leq\; & \frac{c_1}{20}\int_{\mathbb R}\left(u^{2}+u^{2}_{x}\right)\Psi'_{i,K}\,dx+\frac{C}{K}\|u_{0}\|^{3}_{H^{1}(\mathbb R)}\|v_{0}\|_{H^{1}(\mathbb R)}e^{-\frac{1}{K}(\sigma_{0}t+\frac{L}{8})},
\end{align*}
where the Young's inequality is used, the exponential decay of $\Psi'_{i,K}$ on $D^c_{i}$, \eqref{predict-1} and \eqref{convolutionESI}.
Let us tackle now the estimate of $\widetilde{I}_{1,3}$. On $D^c_{i}$ we have
\begin{align*}
&\int_{D^c_{i}}\left[uP_{x}\ast(u^2_{x}v_{x})\right]\Psi'_{i,K}(t)\,dx \leq  \|\Psi'_{i,K}\|_{L^\infty(D^c_{i})}\|v\|_{L^\infty(D^c_{i})}\int_{\mathbb R}u^2_{x}\left[P\ast u\right]\,dx\\
&\quad \leq \;  \|\Psi'_{i,K}\|_{L^\infty(D^c_{i})}\|v\|_{L^\infty(D^c_{i})}\|P\ast u\|_{L^\infty(\mathbb R)}\int_{\mathbb R}u^2_{x}\,dx.
\end{align*}
Applying the Holder inequality, we have for all $x\in \mathbb R$,
\begin{eqnarray}\label{convolutionESI-1}
\begin{aligned}
P\ast u= \frac{1}{2}\int_{\mathbb R}e^{-|x-y|}u(y)\,dy\leq \frac{1}{2}\left(\int_{\mathbb R}e^{-2|x-y|}\,dy\right)^{\frac{1}{2}}\left(\int_{\mathbb R}u^2(y)\,dy\right)^{\frac{1}{2}}\leq \frac{1}{2}\|u\|_{L^2(\mathbb R)},
\end{aligned}
\end{eqnarray}
and then using \eqref{convolutionESI-1} and the exponential decay of $\Psi'_{i,K}$ on $D^c_{i}$, it holds
\begin{align*}
\int_{D^c_{i}}\left[uP_{x}\ast(u^2_{x}v_{x})\right]\Psi'_{i,K}(t)\,dx\leq \frac{C}{K}\|u_{0}\|^{3}_{H^{1}(\mathbb R)}\|v_{0}\|_{H^{1}(\mathbb R)}e^{-\frac{1}{K}(\sigma_{0}t+\frac{L}{8})}.
\end{align*}
The estimate of $\widetilde{I}_{1,3}$ on $D_{i}$ leads to
\begin{align*}
\int_{D_{i}}\left[uP_{x}\ast(u^2_{x}v_{x})\right]\Psi'_{i,K}(t)\,dx \leq & \|u\|_{L^\infty(D_{i})}\|v\|_{L^\infty(D_{i})}\int_{\mathbb R}\left[P\ast\Psi'_{i,K}(t)\right]u^{2}_{x}\,dx\\
\leq&  \frac{c_1}{20}\int_{\mathbb R}\left(u^{2}+u^{2}_{x}\right)\Psi'_{i,K}\,dx.
\end{align*}
Therefore, for $0<\alpha<\alpha_{0}$ and $L>L_{0}>0$, with $\alpha_{0}\ll 1$ and $L_{0}\gg 1$, it holds
\begin{align*}
\frac{d}{dt}\mathcal{J}^{u}_{j,K}(t)\leq \frac{C}{K}\|u_{0}\|^{3}_{H^{1}(\mathbb R)}\|v_{0}\|_{H^{1}(\mathbb R)}e^{-\frac{1}{K}(\sigma_{0}t+\frac{L}{8})}.
\end{align*}
Integrating between $0$ and $t$, we obtain
\begin{align*}
\mathcal{J}^{u}_{j,K}(t)-\mathcal{J}^{u}_{j,K}(0)\leq O(e^{-\frac{\sigma_{0}L}{8K}}).
\end{align*}
Similarly, we also obtain
\begin{align*}
\mathcal{J}^{v}_{j,K}(t)-\mathcal{J}^{v}_{j,K}(0)\leq O(e^{-\frac{\sigma_{0}L}{8K}}).
\end{align*}
Now, we need to prove the third inequality of this proposition. Applying the Virial type identity with $g=\Psi_{i,K}$ and using \eqref{infspeed}, we get
\begin{eqnarray*}
\begin{aligned}
\frac{d}{dt}\mathcal{J}^{u,v}_{j,K}(t)
=&\int_{\mathbb R}\Big[uvu_{x}v_{x}+uP\ast\Big(\frac{1}{2}v^{2}_{x}u+vv_{x}u_{x}+v^{2}u\Big)+uP_{x}\ast\Big(\frac{1}{2}v^{2}_{x}u_{x}\Big)\Big]\Psi'_{i,K}\,dx\\
&\;\;+\int_{\mathbb R}\Big[vP\ast\Big(\frac{1}{2}u^{2}_{x}v+uu_{x}v_{x}+u^{2}v\Big)+uP_{x}\ast\Big(\frac{1}{2}u^{2}_{x}v_{x}\Big)\Big]\Psi'_{i,K}\,dx\\
&\;\;-\dot{y}_{i}\int_{\mathbb R}\left(uv+u_{x}v_{x}\right)\Psi'_{i,K}\,dx\\
=&\int_{\mathbb R}\left[uv(uv+u_{x}v_{x})+uP\ast\left(v(uv+u_{x}v_{x})\right)+vP\ast\left(u(uv+u_{x}v_{x})\right)\right]\Psi'_{i,K}\,dx\\
&\;\;+\int_{\mathbb R}\Big[-\frac{1}{2}u^2v^2+vP\ast\Big(\frac{1}{2}u^2_{x}v\Big)+vP_{x}\ast\Big(\frac{1}{2}u^2_{x}v_{x}\Big)\Big]\Psi'_{i,K}\,dx\\
&\;\;+\int_{\mathbb R}\Big[-\frac{1}{2}u^2v^2+uP\ast\Big(\frac{1}{2}v^2_{x}u\Big)+uP_{x}\ast\Big(\frac{1}{2}v^2_{x}u_{x}\Big)\Big]\Psi'_{i,K}\,dx\\
&\;\;-\dot{y}_{i}\int_{\mathbb R}\left(uv+u_{x}v_{x}\right)\Psi'_{i,K}\,dx\\
\triangleq &\widetilde{I}_{2,1}+\widetilde{I}_{2,2}+\widetilde{I}_{2,3}-\dot{y}_{i}\int_{\mathbb R}\left(uv+u_{x}v_{x}\right)\Psi'_{i,K}\,dx.
\end{aligned}
\end{eqnarray*}
In the same way as the proof of the first inequality, we obtain
\begin{align*}
\widetilde{I}_{2,1} \leq \frac{c_1}{20}\int_{\mathbb R}\left(uv+u_{x}v_{x}\right)\Psi'_{i,K}\,dx+\frac{C}{K}\|u_{0}\|^{2}_{H^{1}(\mathbb R)}\|v_{0}\|^{2}_{H^{1}(\mathbb R)}e^{-\frac{1}{K}(\sigma_{0}t+\frac{L}{8})}.
\end{align*}
For this term $\widetilde{I}_{2,2}$, we note
\begin{align*}
|u_x|\leq u \quad {\rm and} \quad |v_x|\leq v.
\end{align*}
Thus,
\begin{align*}
(u+u_{x})(v+v_{x})\geq 0\quad {\rm and} \quad (u-u_{x})(v-v_{x})\geq 0,
\end{align*}
which is
\begin{eqnarray}\label{oneESI}
\begin{aligned}
|uv_{x}+u_{x}v| \leq |uv+u_{x}v_{x}|.
\end{aligned}
\end{eqnarray}
Thanks to this estimate \eqref{oneESI}, we have
\begin{align*}
\widetilde{I}_{2,2}=&\frac{1}{2}\int_{\mathbb R}\left[(P_{xx}-P)\ast u^2v+P\ast\left(u^2_{x}v\right)+P_{x}\ast\left(u^2_{x}v_{x}\right)\right]v\Psi'_{i,K}\,dx\\
=&\frac{1}{2}\int_{\mathbb R}\left[P\ast \left((u^2_{x}-u^2)v\right)+P_{x}\ast(2uu_{x}v+u^2v_{x}+u^2_{x}v_{x})\right]v\Psi'_{i,K}\,dx\\
\leq &\frac{1}{2}\int_{\mathbb R}\left[P_{x}\ast\left((uv+u_{x}v_{x})u_{x}\right)+P_{x}\ast\left((uv_{x}+u_{x}v)u\right)\right]v\Psi'_{i,K}\,dx\\
\leq &\frac{1}{2}\int_{\mathbb R}\left[P_{x}\ast\left((uv+u_{x}v_{x})u_{x}\right)+P\ast\left((uv+u_{x}v_{x})u\right)\right]v\Psi'_{i,K}\,dx.
\end{align*}
In the same way as proof the inequality of the first inequality, we find
\begin{align*}
\widetilde{I}_{2,2}\leq \frac{c_1}{20}\int_{\mathbb R}\left(uv+u_{x}v_{x}\right)\Psi'_{i,K}\,dx+\frac{C}{K}\|u_{0}\|^{2}_{H^{1}(\mathbb R)}\|v_{0}\|^{2}_{H^{1}(\mathbb R)}e^{-\frac{1}{K}(\sigma_{0}t+\frac{L}{8})}.
\end{align*}
Similarly, we have
\begin{align*}
\widetilde{I}_{2,3}\leq \frac{c_1}{20}\int_{\mathbb R}\left(uv+u_{x}v_{x}\right)\Psi'_{i,K}\,dx+\frac{C}{K}\|u_{0}\|^{2}_{H^{1}(\mathbb R)}\|v_{0}\|^{2}_{H^{1}(\mathbb R)}e^{-\frac{1}{K}(\sigma_{0}t+\frac{L}{8})}.
\end{align*}
Therefore, for $0<\alpha<\alpha_{0}$ and $L>L_{0}>0$, with $\alpha_{0}\ll 1$ and $L_{0}\gg 1$, it holds
\begin{align*}
\frac{d}{dt}\mathcal{J}^{u,v}_{j,K}(t)\leq \frac{C}{K}\|u_{0}\|^{3}_{H^{1}(\mathbb R)}\|v_{0}\|_{H^{1}(\mathbb R)}e^{-\frac{1}{K}(\sigma_{0}t+\frac{L}{8})}.
\end{align*}
Integrating between $0$ and $t$, we obtain
\begin{align*}
\mathcal{J}^{u,v}_{j,K}(t)-\mathcal{J}^{u,v}_{j,K}(0)\leq O(e^{-\frac{\sigma_{0}L}{8K}}).
\end{align*}
This completes the proof of this proposition.
\end{proof}

\subsection{A localized and a global estimate}
We define the function $\Phi_{i}=\Phi_{i}(t,x)$ by $\Phi_{1}=1-\Psi_{2,K}=1-\Psi_{K}(\cdot-y_{2}(t))$, $\Phi_{N}=\Psi_{N,K}=\Psi_{K}(\cdot-y_{N}(t))$ and for $i=2,...,N-1$,
\begin{eqnarray*}
\begin{aligned}
\Phi_{i}=\Psi_{i,K}-\Psi_{i+1,K}=\Psi_{K}(\cdot-y_{i}(t))-\Psi_{K}(\cdot-y_{i+1}(t)),
\end{aligned}
\end{eqnarray*}
where $\Psi_{K}$ and the $y_{i}$ are defined in the previous section. It is easy to check that $\sum^{N}_{i=1}\Phi_{i,K}\equiv1$. We take $L>0$ and $L/K>0$ large enough so that $\Phi_{i}$ satisfies
\begin{eqnarray}\label{weightesi-1}
\begin{aligned}
\left|1-\Phi_{i,K}\right|\leq 4e^{-\frac{L}{8K}} \quad {\rm on}\quad \Big[\tilde{x}_{i}-\frac{L}{4},\tilde{x}_{i}+\frac{L}{4}\Big],
\end{aligned}
\end{eqnarray}
and
\begin{eqnarray}\label{weightesi-2}
\begin{aligned}
\left|\Phi_i\right|\leq 4e^{-\frac{L}{8K}} \quad {\rm on}\quad \Big[\tilde{x}_{j}-\frac{L}{4},\tilde{x}_{j}+\frac{L}{4}\Big]\,\; {\rm whenever}\;\; j\ne i.
\end{aligned}
\end{eqnarray}
We now use the following localized version of $E_{u}$, $E_{v}$, $H$ and $F$ defined for $i\in \{1,...,N\}$, by
\begin{eqnarray}\label{weightesi-3}
\begin{aligned}
&E_{ui}(u)=\int_{\mathbb R}\left(u^2+u^2_{x}\right)\Phi_{i}\,dx, \quad E_{vi}(v)=\int_{\mathbb R}\left(v^2+v^2_{x}\right)\Phi_{i}\,dx, \\
& H_{i}(u,v)=\int_{\mathbb R}\left(uv+u_{x}v_{x}\right)\Phi_{i}\,dx\\
{\rm and}\;\; &F_{i}(u,v)=\int_{\mathbb R}\Big(u^2v^2+\frac{1}{3}u^2v_x^2+\frac{1}{3}v^2u_x^2+\frac{4}{3}uvu_xv_x-\frac{1}{3}u_x^2v_x^2\Big)\Phi_{i}\,dx.
\end{aligned}
\end{eqnarray}
Please note that henceforth we take $K=\sqrt{L}/8$.

The following lemma gives a localized version of \eqref{functionalESI}. Note that the functionals $H_{i}$ and $F_{i}$ do not depend on time in the statement below since we fix $\tilde{x}_{1}<...<\tilde{x}_{N}$.
\begin{lemma}\label{localizedESI}
Let be given $N$ real numbers $\tilde{x}_{1}<...<\tilde{x}_{N}$ with $\tilde{x}_{i}-\tilde{x}_{i-1}\geq 2L/3$. Define the $J_{i}$ as in \eqref{initial-3-4} and assume that, for $i=1,...,N$, there exists $x_{i}\in J_{i}$ such that $|x_{i}-\tilde{x}_{i}|\leq L/12$ and $u(x_{i})v(x_{i})=\max_{x\in J_{i}}uv:=M_{i}$. Then for any $(u,v)\in H^{1}(\mathbb R)\times H^{1}(\mathbb R)$, it holds
\begin{eqnarray}\label{localizedESI-1}
\begin{aligned}
\frac{4}{3}M^2_{i}-\frac{4}{3}M_{i}H_{i}(u,v)+F_{i}(u,v)\leq O(L^{-\frac{1}{2}}),\quad i\in\{1,...,N\}.
\end{aligned}
\end{eqnarray}
\end{lemma}
\begin{proof}
Let $i\in \{1,...,N\}$ be fixed. We introduce the functions $g_{u}$, $g_{v}$ and $h$ defined by
\begin{equation*}
g_{u}(x)=\left\{
\begin{aligned}
&u-u_{x},\quad &x<x_{i},\\
&u+u_{x},&x>x_{i},
\end{aligned}
\right.
\end{equation*}
\begin{equation*}
g_{v}(x)=\left\{
\begin{aligned}
&v-v_{x},\quad &x<x_{i},\\
&v+v_{x},&x>x_{i},
\end{aligned}
\right.
\end{equation*}
and\\
\begin{equation*}
h(x)=\left\{
\begin{aligned}
&uv-\frac13 (uv)_x-\frac{1}{3}u_{x}v_{x},\quad &x<x_{i},\\
&uv+\frac 13 (uv)_x+\frac{1}{3}u_{x}v_{x},&x>x_{i}.
\end{aligned}
\right.
\end{equation*}
Integrating by parts we compute
\begin{eqnarray}\label{localizedfuncESI-1}
\begin{aligned}
&\int_{\mathbb R}h(x)g_{u}(x)g_{v}(x)\Phi_{i}dx\\
=&\int_{-\infty}^{x_i}{\Big(uv-\frac{1}{3}(uv)_x-\frac{1}{3}u_{x}v_{x}\Big)\Big(uv-(uv)_x+u_{x}v_{x}\Big)}\Phi_{i}\,dx\\
&\qquad +\int_{x_i}^{\infty}{\Big(uv+\frac 13 (uv)_x-\frac{1}{3}u_{x}v_{x}\Big)\Big(uv+(uv)_x+u_xv_x\Big)}\Phi_{i}\,dx\\
=&\int_{-\infty}^{\infty}{\Big(u^2v^2+\frac{1}{3}u^2v_x^2+\frac{1}{3}v^2u_x^2+\frac{4}{3}uvu_xv_x-\frac{1}{3}u_x^2v_x^2\Big)}\,dx-\frac{4}{3}\,M_{i}^2\Phi_{i}(x_{i})\\
&\qquad+\frac{2}{3}\int^{x_{i}}_{\-\infty}u^2v^2\Phi'_{i}\,dx-\frac{2}{3}\int^{\-\infty}_{x_{i}}u^2v^2\Phi'_{i}\,dx.\\
\end{aligned}
\end{eqnarray}
Recall that we take $K=\sqrt{L}/8$ and thus $|\Phi'|\leq C/K=O(L^{-1/2})$. Moreover, since $|x_{i}-\tilde{x}_{i}|\leq L/12$, it follows from \eqref{weightesi-1} that $\Phi_{i}(x_{i})=1+O(e^{-\sqrt{L}})$ and thus
\begin{eqnarray}\label{localizedfuncESI-2}
\begin{aligned}
\int_{\mathbb R}h(x)g_{u}(x)g_{v}(x)\Phi_{i}dx=&F_{i}(u,v)-\frac{4}{3}M^2_{i}+\|u_{0}\|^2_{H^{1}(\mathbb R)}\|v_{0}\|^2_{H^{1}(\mathbb R)}O(L^{-\frac{1}{2}})+O(e^{-\sqrt{L}}).
\end{aligned}
\end{eqnarray}
On the other hand, we firstly claim that $x_{1}$ is the maximum point of the function $u(x,t)v(x,t)$ on $(-\infty,y_{2}(t)+aL]$, $x_{N}$ is the maximum point of the function $u(x,t)v(x,t)$ on $[y_{N}(t)-aL,+\infty)$ and $x_{i}$ is the maximum point of the function $u(x,t)v(x,t)$ on $[y_{i}(t)-aL,y_{i+1}(t)+aL]$, where $a$ is a constant chosen later and $i\in \{2,...,N-1\}$.

We need to show  $x_{i}$ is the maximum point of the function $u(x,t)v(x,t)$ on $[y_{i}(t)-aL,y_{i}(t)]\,\cup\, [y_{i+1}(t),y_{i+1}(t)+aL]$, where $i\in \{2,...,N-1\}$.
If $x\in [y_{i}(t)-aL,y_{i}(t)]\,\cup\, [y_{i+1}(t),y_{i+1}(t)+aL]$, then
\begin{eqnarray*}
\begin{aligned}
u(x,t)v(x,t)\leq c_{i+1}e^{-\left(\frac{3}{8}-a\right)L}+c_{i}e^{-\frac{3}{8}L}+O(\sqrt{\alpha})+O(e^{-\frac{L}{4}}).
\end{aligned}
\end{eqnarray*}
Choosing $a$ with
\begin{eqnarray*}
\begin{aligned}
c_{i+1}e^{-\left(\frac{3}{8}-a\right)L}+c_{i}e^{-\frac{3}{8}L}+O(\sqrt{\alpha})+O(e^{-\frac{L}{4}})\leq c_{i}-O(\sqrt{\alpha})-O(e^{-\frac{L}{4}}),
\end{aligned}
\end{eqnarray*}
and using the same estimate as above, we have proved similar conclusion of $x_{1}$ and $x_{N}$.

Next, we define the intervals $\hat{J}_{1}=(-\infty,y_{2}(t)+aL]$, $\hat{J}_{N}=[y_{N}(t)-aL,+\infty)$ and for $i=2,...,N-1$, $\hat{J}_{i}=[y_{i}(t)-aL,y_{i+1}(t)+aL]$, where $a$ is chosen above.
\begin{eqnarray}
\begin{aligned}
\int_{\mathbb R}h(x)g_{u}(x)g_{v}(x)\Phi_{i}dx=&\int_{\hat{J}_{i}}h(x)g_{u}(x)g_{v}(x)\Phi_{i}dx+\int_{\hat{J}^c_{i}}h(x)g_{u}(x)g_{v}(x)\Phi_{i}dx\\
\leq & \frac{4}{3}M_{i}\int_{\mathbb R}g_{u}(x)g_{v}(x)\Phi_{i}dx+\int_{\hat{J}^c_{i}}h(x)g_{u}(x)g_{v}(x)\Phi_{i}dx,
\end{aligned}
\end{eqnarray}
where we use the fact $h(x)\leq 4u(x)v(x)/3\leq 4M_{i}/3$ on the $\hat{J}_{i}$. One chooses $L$ large enough satisfying $aL>10\sqrt{L}$, which leads to
\begin{eqnarray*}
\begin{aligned}
\int_{\hat{J}^c_{i}}h(x)g_{u}(x)g_{v}(x)\Phi_{i}dx\leq C\|u_{0}\|^2_{H^{1}(\mathbb R)}\|v_{0}\|^2_{H^{1}(\mathbb R)}e^{-\tilde{a}L}\leq O(L^{-1}),
\end{aligned}
\end{eqnarray*}
where $C$ and $\tilde{a}$ are constants. Therefore,
\begin{eqnarray*}
\begin{aligned}
\frac{4}{3}M^2_{i}-\frac{4}{3}M_{i}H_{i}(u,v)+F_{i}(u,v)\leq O(L^{-\frac{1}{2}}),\quad i\in\{1,...,N\}.
\end{aligned}
\end{eqnarray*}
This completes the proof.
\end{proof}
Now let us state a global identity related to \eqref{energyESI}.
\begin{lemma}\label{globalIDE}
For any $Z\in \mathbb R^N$ satisfying $|z_{i}-z_{i-1}|\geq L/2$ and any $(u,v)\in H^{1}(\mathbb R)\times H^{1}(\mathbb R)$, there holds
\begin{eqnarray*}
\begin{aligned}
&E_{u}(u)-\sum^{N}_{i=1}E_{u}(\varphi_{c_{i}})=\|u-R_{Z}\|^2_{H^1}+4\sum^N_{i=1}a_{i}(u(z_{i})-a_{i})+O\left(e^{-\frac{L}{4}}\right),\\
&E_{v}(v)-\sum^{N}_{i=1}E_{v}(\psi_{c_{i}})=\|v-S_{Z}\|^2_{H^1}+4\sum^N_{i=1}b_{i}(v(z_{i})-b_{i})+O\left(e^{-\frac{L}{4}}\right).
\end{aligned}
\end{eqnarray*}
\end{lemma}
\begin{proof}
Using the relation between $\varphi$ and its derivative, integrating by parts, we get
\begin{eqnarray*}
\begin{aligned}
&E_{u}(u-R_{Z})=E_{u}(u)+E_{u}(R_{Z})-2\sum^{N}_{i=1}\int_{\mathbb R}u\varphi_{c_{i}}(\cdot-z_{i})+u_{x}\partial_{x}\varphi_{c_{i}}(\cdot-z_{i})\\
&=E_{u}(u)+E_{u}(R_{Z})-2\sum^{N}_{i=1}\int_{\mathbb R}u\varphi_{c_{i}}(\cdot-z_{i})+2\sum^{N}_{i=1}\int^{\infty}_{z_{i}}u_{x}\varphi_{c_{i}}(\cdot-z_{i})-2\sum^{N}_{i=1}\int^{z_{i}}_{-\infty}u_{x}\varphi_{c_{i}}(\cdot-z_{i})\\
&=E_{u}(u)+E_{u}(R_{Z})-4\sum^{N}_{i=1}a_{i}u(z_{i}).
\end{aligned}
\end{eqnarray*}
On the other hand, since $|z_{i}-z_{i-1}|\geq L/2$, it is easy to check that
\begin{eqnarray*}
\begin{aligned}
E(R_{Z})=\sum^{N}_{i=1}E(\varphi_{c_{i}})+O\left(e^{-\frac{L}{4}}\right).
\end{aligned}
\end{eqnarray*}
Combining these two identities, the desired result follows, and similarly, we obtain the second equality of this lemma.
\end{proof}

\subsection{End of the proof of Theorem 1.2}

Before we give the final proof, we need to prove the following lemmas.
\begin{lemma}\label{FinalESI}
Assume $\|u_{0}-R_{Z^{0}}\|_{H^{1}}+\|v_{0}-S_{Z^{0}}\|_{H^{1}}< \epsilon $. Then for any $i=\{1,...,N\}$, we have
\begin{eqnarray*}
\begin{aligned}
&\left|H_{i}(u_{0},v_{0})-H_{i}(\varphi_{c_{i}},\psi_{c_{i}})\right|<O(\epsilon)+O\left(e^{-\sqrt{L}}\right), \\
&\left|E_{ui}(u_{0})-E_{ui}(\varphi_{c_{i}})\right|<O(\epsilon)+O\left(e^{-\sqrt{L}}\right),\\
&\left|E_{vi}(v_{0})-E_{vi}(\psi_{c_{i}})\right|<O(\epsilon)+O\left(e^{-\sqrt{L}}\right).
\end{aligned}
\end{eqnarray*}
\end{lemma}
\begin{proof}
Define the interval $\bar{J}_{i}=\left[z^0_{i}-L/4, \,z^0_{i}+L/4\right]$. Using \eqref{weightesi-1}, \eqref{weightesi-2} and the exponential decay of $\varphi_{c_{i}}$'s and the $\psi_{c_{i}}$'s, we have
\begin{eqnarray*}
\begin{aligned}
&\left|H_{i}(\varphi_{c_{i}},\psi_{c_{i}})-H_{i}(u_{0},v_{0})\right|=\left|\int_{\mathbb R}\varphi_{c_{i}}\psi_{c_{i}}+\varphi'_{c_{i}}\psi'_{c_{i}}\,dx-\int_{\mathbb R}\left(u_{0}v_{0}+u_{0x}v_{0x}\right)\Phi_{i}\,dx\right|\\
\leq &\left|\int_{\bar{J}_{i}}\varphi_{c_{i}}\psi_{c_{i}}+\varphi'_{c_{i}}\psi'_{c_{i}}\,dx-\int_{\bar{J}_{i}}\left(u_{0}v_{0}+u_{0x}v_{0x}\right)\,dx\right|+O\left(e^{-\sqrt{L}}\right)+O\left(e^{-\frac{L}{8}}\right).
\end{aligned}
\end{eqnarray*}
Since $\|u_{0}-R_{Z}\|_{H^1}+\|v_{0}-S_{Z}\|_{H^1}<\epsilon$, then we obtain
\begin{eqnarray*}
\begin{aligned}
&\|u_{0}-\varphi_{c_{i}}\|_{H^1(\bar{J}_{i})}+\|v_{0}-\psi_{c_{i}}\|_{H^1(\bar{J}_{i})}\\
&\leq |u_{0}-R_{Z}\|_{H^1(\bar{J}_{i})}+\|v_{0}-S_{Z}\|_{H^1(\bar{J}_{i})}+\sum_{i\ne j}\left(\|\varphi_{c_{j}}\|_{H^1(\bar{J}_{i})}+\|\psi_{c_{j}}\|_{H^1(\bar{J}_{i})}\right)\\
&\leq |u_{0}-R_{Z}\|_{H^1(\mathbb R)}+\|v_{0}-S_{Z}\|_{H^1(\mathbb R)}+O\left(e^{-\frac{L}{8}}\right)<  \epsilon +O\left(e^{-\frac{L}{8}}\right).
\end{aligned}
\end{eqnarray*}
Using the same way in proving the third inequality of \eqref{lem1.3}, we obtain
\begin{eqnarray*}
\begin{aligned}
\left|H_{i}(\varphi_{c_{i}},\psi_{c_{i}})-H_{i}(u_{0},v_{0})\right|<O(\epsilon)+O\left(e^{-\sqrt{L}}\right).
\end{aligned}
\end{eqnarray*}
Similarly,
\begin{eqnarray*}
\left|E_{ui}(\varphi_{c_{i}})-E_{ui}(u_{0})\right|<O(\epsilon)+O\left(e^{-\sqrt{L}}\right)\; {\rm and} \; \left|E_{vi}(\psi_{c_{i}})-E_{vi}(v_{0})\right|<O(\epsilon)+O\left(e^{-\sqrt{L}}\right).
\end{eqnarray*}
\end{proof}
\begin{lemma}\label{productESI}
Assume $u\in H^s$,$v\in H^s$, $s>3$ and $u\geq 0$, $v\geq 0$. Let
\begin{align*}
M_{i}=u(x_{i})v(x_{i})=\max_{x\in J_{i}}u(x)v(x),
\end{align*}
and set $\alpha_{0}=A({\epsilon_0}^{1/4}+L_{0}^{-1/8})$. Then we have
\begin{align*}
|M_{i}-c_{i}|\leq O\big(L^{-\frac{1}{4}}\big)+O\big(\epsilon^{\frac{1}{2}}\big),\quad  i=1,...,N.
\end{align*}
\end{lemma}
\begin{proof}
Define the second-order polynomials $\bar{P}^{i}$ and $\hat{P}^{i}$ by:
\begin{align*}
\bar{P}^{i}(y)=y^2-H_{i}(u,v)y+\frac{3}{4}F_{i}(u,v) \;\; \mathrm{and} \;\; \hat{P}^{i}(y)=y^2-H_{i}(u_0,v_0)y+\frac{3}{4}F_{i}(u_{0},v_{0}).
\end{align*}
For the peaked solution, $H(\varphi_{c_{i}},\psi_{c_{i}})=2c_{i}$ and $F(\varphi_{c_{i}},\psi_{c_{i}})=4c^2_{i}/3$, the above polynomial becomes
\begin{eqnarray}\label{finalESI-1}
\begin{aligned}
&\hat{P}^{i}_{0}(y)=y^2-2c_{i}y+c^2_{i}=(y-c_{i})^2.
\end{aligned}
\end{eqnarray}
Since
\begin{align*}
\bar{P}^{i}(y)=&\hat{P}^{i}_{0}(y)+\left(H(\varphi_{c_{i}},\psi_{c_{i}})-H_{i}(u_0,v_0)\right)y+\left(H_{i}(u_0,v_0)-H_{i}(u,v)\right)y\\
&\qquad +\frac{3}{4}\left[\left(F_{i}(u_{0},v_{0})-F(\varphi_{c_{i}},\psi_{c_{i}})\right)+\left(F_{i}(u,v)-F_{i}(u_{0},v_{0})\right)\right],
\end{align*}
it follows that
\begin{eqnarray*}
\begin{aligned}
\sum^{N}_{i=1}\bar{P}^{i}&(M_{i})=\sum^{N}_{i=1}\hat{P}^{i}_{0}(M_{i})+\sum^{N}_{i=1}M_{i}\left(H(\varphi_{c_{i}},\psi_{c_{i}})-H_{i}(u_0,v_0)\right)\\
&+\frac{4}{3}\Big(F(u_{0},v_{0})-\sum^{N}_{i=1}F(\varphi_{c_{i}},\psi_{c_{i}})\Big)+\sum^N_{i=1}M_{i}\Big((H_{i}(u_0,v_0)-H_{i}(u,v)\Big)\leq O(L^{-\frac{1}{2}}).
\end{aligned}
\end{eqnarray*}
Using the Abel transformation of the last term and the monotonicity property and combining Lemma \eqref{FinalESI}, we thus get
\begin{align*}
\sum^{N}_{i=1}(M_{i}-c_{i})^2\leq &\sum^{N}_{i=1}M_{i}\left|H(\varphi_{c_{i}},\psi_{c_{i}})-H_{i}(u_0,v_0)\right|+\frac{3}{4}\Big|F(u_{0},v_{0})-F(R_{Z},S_{Z})\Big|+O\big(L^{-\frac{1}{2}}\big)+O(\epsilon)\\
\leq & O\big(L^{-\frac{1}{2}}\big)+O(\epsilon),
\end{align*}
which completes the proof of the lemma.
\end{proof}

\begin{lemma}\label{FLGJ}
Assume $|M_{i}-c_{i}|\leq O\left(L^{-1/4}\right)+O\left(\epsilon^{1/2}\right),\quad i=1,...,N$. Then for any $i=\{1,...,N\}$, we have
\begin{align*}
|u(x_{i})-a_{i}|\leq O\left(L^{-1/4}\right)+O\Big(\epsilon^{1/2}\Big)\quad {\rm and} \quad |v(x_{i})-b_{i}|\leq O\left(L^{-1/4}\right)+O\left(\epsilon^{1/2}\right).
\end{align*}
\end{lemma}
Note the monotonicity property that $\mathcal{J}^{u}_{j,K}(t)$ is close to $\left\|u(t)\right\|_{H^{1}(x>y_{j}(t))}$ and thus measures the energy at the right of the $(j-1)th$ bump of $u$ and  $\mathcal{J}^{v}_{j,K}(t)$ is close to $\left\|v(t)\right\|_{H^{1}(x>y_{j}(t))}$ and thus measures the energy at the right of the $(j-1)th$ bump of $v$. We thus use the induction argument to prove this lemma.
\begin{proof}
{Step 1.} 
Define
\begin{equation}
g_{uN}(x)=
\left\{
\begin{aligned}
&u(x)-u_{x}(x), \quad x<x_N,\\
&u(x)+u_{x}(x), \quad x>x_N.
 \end{aligned}
\right.
\end{equation}
Then we have
\begin{eqnarray*}
\begin{aligned}
0\leq \int_{\mathbb R}g^2_{uN}(x)\Phi_{N}(x)\,dx=&\int_{\mathbb R}\left(u^2+u^2_{x}\right)\Phi_{N}(x)\,dx-2u^{2}(x_{N})\Phi_{N}(x_N)\\
&+\int^{x_{N}}_{-\infty}u^2\Phi'_{N}\,dx-\int^{+\infty}_{x_{N}}u^2\Phi'_{N}\,dx.
\end{aligned}
\end{eqnarray*}
Using $|x_{i}-\tilde{x}_{i}|\leq L/12$, it follows from \eqref{weightesi-1} that $\Phi_{i}(x_{i})=1+O(e^{-\sqrt{L}})$ and thus
\begin{align*}
E_{uN}(u(t))-2u^2(x_{N})+O\left(L^{-\frac{1}{2}}\right)\geq 0.
\end{align*}
Note that
\begin{eqnarray*}
E_{uN}(u(t))=\mathcal{J}^u_{N,K}(t)\quad {\rm and} \quad E_{uN}(u(t))-E_{uN}(u(0))\leq O\big(e^{-\sigma_{0}\sqrt{L}}\big).
\end{eqnarray*}
We thus get
\begin{eqnarray*}
\begin{aligned}
u^2(x_{N})\leq O(L^{-\frac{1}{2}})+O(\epsilon)+a^2_{N}.
\end{aligned}
\end{eqnarray*}
Similarly,
\begin{eqnarray*}
\begin{aligned}
v^2(x_{N})\leq O(L^{-\frac{1}{2}})+O(\epsilon)+b^2_{N}.
\end{aligned}
\end{eqnarray*}
Combining the identity $|M_{N}-c_{N}|\leq O\left(L^{-1/4}\right)+O\left(\epsilon^{1/2}\right)$, we arrive at
\begin{eqnarray*}
\begin{aligned}
|u(x_{N})-a_{N}|\leq O\big(L^{-\frac{1}{4}}\big)+O\big(\epsilon^{\frac{1}{2}}\big) \quad {\rm and} \quad |v(x_{N})-b_{N}|\leq O\big(L^{-\frac{1}{4}}\big)+O\big(\epsilon^{\frac{1}{2}}\big).
\end{aligned}
\end{eqnarray*}
{Step 2.}  Assume for any $k\leq i<N$, we have
\begin{align*}
|u(x_{i})-a_{i}|\leq O\big(L^{-\frac{1}{4}}\big)+O\big(\epsilon^{\frac{1}{2}}\big) \quad {\rm and} \quad |v(x_{i})-b_{i}|\leq O\big(L^{-\frac{1}{4}}\big)+O\big(\epsilon^{\frac{1}{2}}\big).
\end{align*}
We need to prove
\begin{align*}
|u(x_{i-1})-a_{i-1}|\leq O\big(L^{-\frac{1}{4}}\big)+O\big(\epsilon^{\frac{1}{2}}\big) \quad {\rm and} \quad |v(x_{i-1})-b_{i-1}|\leq O\big(L^{-\frac{1}{4}}\big)+O\big(\epsilon^{\frac{1}{2}}\big).
\end{align*}
Firstly, we have
\begin{align*}
\sum^{N}_{i=k-1}u^2(x_{i})\leq &\;\frac{1}{2}\sum^{N}_{i=k-1}E_{ui}(u(t))+O(L^{-\frac{1}{2}})\\
\leq &\; \frac{1}{2}\left(\mathcal{J}^u_{k-1,K}(t)-\mathcal{J}^u_{k-1,K}(0)\right)+\frac{1}{2}\Big(\mathcal{J}^u_{k-1,K}(0)-\sum^{N}_{i=k-1}E_{u}(\varphi_{c_{i}})\Big)\\
&\quad + \frac{1}{2}\sum^{N}_{i=k-1}E_{u}(\varphi_{c_{i}})+O\Big(L^{-\frac{1}{2}}\Big).
\end{align*}
It follows that
\begin{align*}
\sum^{N}_{i=k-1}u^2(x_{i})\leq \frac{1}{2}\sum^{N}_{i=k-1}E_{u}(\varphi_{c_{i}})+O\Big(L^{-\frac{1}{2}}\Big)+O(\epsilon).
\end{align*}
Using the assumption, we find
\begin{align*}
u^2(x_{k-1})\leq a^2_{k-1}+O\big(L^{-\frac{1}{4}}\big)+O\left(\sqrt{\epsilon}\right).
\end{align*}
Similarly
\begin{align*}
v^2(x_{k-1})\leq b^2_{k-1}+O\big(L^{-\frac{1}{4}}\big)+O\left(\sqrt{\epsilon}\right).
\end{align*}
Those together with the identity \eqref{productESI} imply that for any $i=1,\cdots ,N$
\begin{align*}
|u(x_{i})-a_{i}|\leq O\big(L^{-\frac{1}{4}}\big)+O\big(\epsilon^{\frac{1}{2}}\big) \quad {\rm and} \quad |v(x_{i})-b_{i}|\leq O\big(L^{-\frac{1}{4}}\big)+O\big(\epsilon^{\frac{1}{2}}\big).
\end{align*}
\end{proof}

\begin{proof} [End of the Proof of Theorem 1.2]
To conclude the proof, from the \eqref{globalIDE}, it thus suffices to prove that there exists $C>0$ which does not depend on $A$ such that
\begin{eqnarray*}
\begin{aligned}
|u(x_{N})-a_{N}|\leq C\Big(L^{-\frac{1}{4}}+\epsilon^{\frac{1}{2}}\Big) \quad {\rm and} \quad |v(x_{N})-b_{N}|\leq C\Big(L^{-\frac{1}{4}}+\epsilon^{\frac{1}{2}}\Big),
\end{aligned}
\end{eqnarray*}
which has been verified by \eqref{FLGJ}.
From \eqref{initial-3-3} and \eqref{initial-3-5}, we know that for $i=2,...,N$,
\begin{eqnarray*}
\begin{aligned}
x_{i}-x_{i-1}\geq \frac{2}{3}L.
\end{aligned}
\end{eqnarray*}
This completes the proof of the theorem.
\end{proof}

\vglue .5cm

\vskip 0.2cm
\noindent {\bf Acknowledgments.} The work of He is partially supported by the NSF-China grant-11971251.  The work of Liu is partially supported by the NSF-China grant-12271424 and  grant-11871395. The work of Qu is partially supported by the NSF-China grant-11631007 and grant-11971251.

\end{document}